\documentclass[11pt]{amsart}

\usepackage{amsfonts, amstext, amsmath, amsthm, amscd, amssymb}
\usepackage{graphicx, color}
\usepackage{import}

\usepackage{microtype}

\usepackage[margin=3cm]{geometry}

\usepackage[hidelinks,pagebackref,pdftex]{hyperref}

\AtBeginDocument{%
   \def\MR#1{}
}

\usepackage{marginnote}
\long\def\@savemarbox#1#2{\global\setbox#1\vtop{\hsize\marginparwidth 
  \@parboxrestore\tiny\raggedright #2}}
\newtheorem{theorem}{Theorem}[section]
\newtheorem{proposition}[theorem]{Proposition}
\newtheorem{lemma}[theorem]{Lemma}
\newtheorem{corollary}[theorem]{Corollary}

\newtheorem*{namedtheorem}{\theoremname}
\newcommand{\theoremname}{testing}
\newenvironment{named}[1]{\renewcommand{\theoremname}{#1}\begin{namedtheorem}}{\end{namedtheorem}}

\theoremstyle{definition}
\newtheorem{definition}[theorem]{Definition}
\newtheorem{example}[theorem]{Example}
\newtheorem{conjecture}[theorem]{Conjecture}
\newtheorem{remark}[theorem]{Remark}
\newtheorem{question}[theorem]{Question}

\newcommand{\refthm}[1]{Theorem~\ref{Thm:#1}}
\newcommand{\reflem}[1]{Lemma~\ref{Lem:#1}}
\newcommand{\refprop}[1]{Proposition~\ref{Prop:#1}}
\newcommand{\refcor}[1]{Corollary~\ref{Cor:#1}}

\newcommand{\refconj}[1]{Conjecture~\ref{Conj:#1}}

\newcommand{\refdef}[1]{Definition~\ref{Def:#1}}
\newcommand{\refsec}[1]{Section~\ref{Sec:#1}}
\newcommand{\reffig}[1]{Figure~\ref{Fig:#1}}

\newcommand{\CC}{\mathbb{C}}  
\newcommand{\RR}{\mathbb{R}}  
\newcommand{\bdy}{\partial}   

\newcommand{\from}{\colon\thinspace} 

\begin{document}

\title{Satellites and Lorenz knots}
\author{Thiago de Paiva}
\author{Jessica S. Purcell}

\begin{abstract}
We construct infinitely many families of Lorenz knots that are satellites but not cables, giving counterexamples to a conjecture attributed to Morton. We amend the conjecture to state that Lorenz knots that are satellite have companion a Lorenz knot, and pattern equivalent to a Lorenz knot. We show this amended conjecture holds very broadly: it is true for all Lorenz knots obtained by high Dehn  filling on a parent link, and other examples. 
\end{abstract}

\maketitle

\section{Introduction}

Lorenz knots and links were initially described in applied dynamical systems: They are the closed periodic orbits of a system of ordinary differential equations studied by meteorologist E.~N.~Lorenz to predict weather patterns~\cite{Lorenz}. These were given a geometric framework by Guckenheimer and Williams~\cite{GuckenheimerWilliams}, proved later by Tucker~\cite{Tucker}, who proved such links are equivalent to links on an embedded branched surface in $\RR^3$ called the Lorenz template.
Birman and Williams used the template to initiate the systematic study of such links from a more knot theoretic point of view~\cite{BirmanWilliams}. Swinging back to dynamics, Ghys showed that Lorenz knots coincide with periodic orbits in the geodesic flow on the modular surface~\cite{Ghys}. 

Such knots appear again in hyperbolic geometry. Birman and Kofman discovered that of the hyperbolic knots that appear in the SnapPy census, triangulated by at most seven tetrahedra, over half of these knots are Lorenz knots~\cite{BirmanKofman}. This is in spite of the fact that Lorenz knots are quite scarce in tables of knots indexed by crossing, observed by Ghys and Leys~\cite{GhysLeys}. In particular, this means that such knots are among those hyperbolic knots of smallest volume. Birman and Kofman ask, why are so many simple hyperbolic knots equivalent to Lorenz knots?

A related question also needs to be answered, namely, when is a Lorenz knot hyperbolic? Thurston showed that any knot in the 3-sphere has complement that is either hyperbolic, a torus knot, or a satellite knot \cite{Thurston:Bulletin}. Which of these are Lorenz knots?

Birman and Williams showed that every torus knot is a Lorenz knot \cite[Theorem~6.1]{BirmanWilliams}. They showed that the satellites obtained as certain cables of Lorenz knots are also Lorenz knots~\cite[Theorem~6.2]{BirmanWilliams}. This was extended by El-Rifai, who showed that the only way in which a Lorenz knot can be presented as the satellite of a Lorenz knot is if it is a cable on a Lorenz knot, possibly with additional twisting~\cite[Theorem~3.9]{ElRifai:Satellite}. In the survey article~\cite{Dehornoy:Survey}, Dehornoy presents Conjecture~5.2 attributed to Morton: That every Lorenz knot that is a satellite is a cable on a Lorenz knot.

In this paper, we show that as stated, this conjecture is false. We construct infinitely many examples of Lorenz knots that are satellites but are not cables on Lorenz knots. 

\begin{theorem}\label{Thm:SatelliteLorenzIntro}
There exist infinitely many Lorenz knots that are satellites for which there are exactly two components of the JSJ decomposition, and both components are hyperbolic. Thus by the uniqueness of the JSJ decomposition of a knot complement, the knots cannot be cables on Lorenz knots. 
\end{theorem}

Recall that a \emph{satellite knot} is built by starting with a knot $P$ in $S^1\times D^2$ that is not isotopic into a ball in $S^1\times D^2$, and embedding it in the tubular neighbourhood $N(C)$ of a knot $C$ in $S^3$ via a homeomorphism $f\from D^2\times S^1\to N(C)$. The \emph{satellite} $K$ is $f(P)$. The knot $P$ is called the \emph{pattern}, and $C$ is called the \emph{companion}. We also say that $K$ is a satellite of $C$. Finally, we do not allow $f$ to introduce additional twisting: we require $f$ to send the $0$-framed longitude $S^1\times\{1\}$ of $S^1\times D^2$ to the standard longitude of $N(C)$. 

When $P$ is  isotopic to a $(p,q)$-torus knot on a torus parallel to the boundary of $S^1\times D^2$, we say that $K$ is a \emph{cable knot}, or $K$ is the $(p,q)$-cable on $C$. 

The Lorenz knots of \refthm{SatelliteLorenzIntro} have companion a Lorenz knot and pattern within a class of twisted torus knots that are equivalent to Lorenz knots, but the pattern cannot be a torus knot. Thus we do not have a counterexample to the part of Morton's conjecture concerning the companion, but only the part concerning the pattern: these knots are not cables. 
Birman and Kofman ask a more general question than Morton's conjecture~\cite[Question~5]{BirmanKofman}: Can a Lorenz knot be a satellite of a non-Lorenz knot? We have no evidence that this is not the case. 

Indeed, we give additional evidence in this paper that Lorenz knots can only be satellites of other Lorenz knots. We upgrade Birman and Kofman's question to a conjecture, modifying Morton's conjecture.

\begin{conjecture}[Lorenz satellite conjecture]\label{Conj:LorenzSatellite}
A Lorenz knot that is a satellite has companion a Lorenz knot. Its pattern, when embedded in an unknotted solid torus in $S^3$, is equivalent to a Lorenz knot in $S^3$.
\end{conjecture}

It would be useful to classify all satellite, torus, and hyperbolic Lorenz knots. In terms of hyperbolic Lorenz knots, some work has been done to identify families by Gomes, Franco and Silva~\cite{GomesFrancoSilva:Partial, GomesFrancoSilva:Farey}. 
Once a Lorenz knot is known to be hyperbolic, various geometric properties can be studied, such as volumes, for example~\cite{CFKNP, Rodriguez-Migueles}.

We will study satellite Lorenz links from the point of view of T-links, introduced by Birman and Kofman~\cite{BirmanKofman}.
For $2\leq r_1< \dots < r_k$, and all $s_i>0$, the T-link $T((r_1,s_1), \dots, (r_k,s_k))$ is defined to be the closure of the following braid:
\[ (\sigma_1\sigma_2\dots\sigma_{r_1-1})^{s_1}(\sigma_1\sigma_2\dots\sigma_{r_2-1})^{s_2}\dots(\sigma_1\sigma_2\dots\sigma_{r_k-1})^{s_k}.\]
Here $\sigma_i$ is a standard generator of the braid group, giving a negative crossing between the $i$-th and $(i+1)$-th strands.
Birman and Kofman showed that T-links exactly coincide with Lorenz links~\cite[Theorem~1]{BirmanKofman}.

Note that T-links include some well-known knots. When $k=2$, and $s_1=sr_1$ is a multiple of $r_1$, the T-knot $T((r_1,sr_1),(r_2,s_2))$ is the twisted torus knot $K(r_2,s_2;r_1,s)$, introduced by Dean~\cite{Dean}.
The hyperbolicity of twisted torus knots has been studied extensively, particularly by Lee, who has classified those with $s>0$ that are torus knots~\cite{Lee:PosTwistedTorus-Torus, Lee:2018Satellite}, and has classified those that are satellites for $|s|>1$~\cite{Lee:2018Satellite}. Hence we have a complete picture of the geometry of twisted torus knots in these cases. 

In this paper, we give large families of satellite knots and links that are T-links. 

\begin{theorem}\label{Thm:SatelliteIntro}
  Let $a_1, b_1, \dots, a_n, b_n, s_1, t_1, \dots, s_m, t_m, p,q$ be integers satisfying:
  \[ 1<a_1<\dots<a_n<q < qs_1<\dots<qs_m < p, \mbox{ and } b_i, t_i >0 \mbox{ for all $i$}. \]
Then the T-link
\[ T((a_1,a_1b_1), \dots, (a_n,a_nb_n),(s_1q,s_1qt_1), \dots, (s_mq, s_mqt_m),(p,q)) \]
is a satellite link with companion the T-link
\[ T((s_1,t_1s_1), \dots, (s_m,t_ms_m+1)) \]
and pattern the T-link
\[ T((a_1,a_1b_1),\dots,(a_n,a_nb_n),(q,p+s_1^2qt_1+\dots+s_m^2qt_m)).\]
\end{theorem}

We actually prove something slightly stronger, namely in  \refthm{SatelliteMultTwists} we show that we have a satellite independent of whether the first $n$ entries are full-twists or not. That is, \refthm{SatelliteIntro} is stated for T-links containing only braids that are full-twists $(a_i,a_ib_i)$ for $i=1, \dots n$, but a similar result is true for $(a_i,b_i)$ with the second entry not a multiple of $a_i$, i.e.\ not a full twist. We still require full-twists in the entries $(qs_i,qs_it_i)$. 

Our theorems so far state that certain knots and links must be satellite. To prove \refconj{LorenzSatellite} and characterise all satellite Lorenz knots, we need to know that other knots are \emph{not} satellite. While we cannot do this in full generality, again we are able to make some progress in the case of full twists.

\begin{named}{\refcor{FullTwistConj}}
Let $p,q$ be relatively prime integers with $1<q<p$, and let $a_1, \dots, a_n$ and $b_1, \dots, b_n$ be integers such that $1<a_1<\dots<a_n<p$ and $b_i>0$, and no $a_i$ is a multiple of $q$. Then there exists $B \gg 0$ such that if each $b_i>B$, then
$T((a_1, a_1b_1), \dots, (a_n,a_nb_n), (p,q))$ is hyperbolic.
\end{named}

So far, our results concerning Lorenz knots that are satellite knots give evidence for the Lorenz satellite conjecture, \refconj{LorenzSatellite}, but they only apply to T-links obtained by full twisting. We also give a family of T-links not obtained by adding full twists to portions of torus links in any obvious manner, but are still satellite, and still satisfy \refconj{LorenzSatellite}.
These are constructed to have a fixed companion that is a torus knot $T(c, c-1)$. 

\begin{named}{\refthm{NoTwistToroidal}}
  Choose positive integers as follows. Let $c\geq 3$, $r\geq 2$, $1\leq k\leq r-1$. Let $2\leq a_1 < \dots < a_n \leq r-k$, and let $b_1, \dots, b_n>0$. Then the T-link
  \[ T((a_1,b_1),\dots,(a_n,b_n),(rc-k, r-k), (rc, r(c-2)+k)) \]
  is a satellite link with companion the torus knot $T(c,c-1)$ and pattern the T-link
  \[ T((a_1,b_1), \dots, (a_n,b_n),(r-k, r-k),(r,r(c-1)^2+r(c-2)+k)). \qedhere \]
\end{named}

Theorems~\ref{Thm:SatelliteIntro} and~\ref{Thm:NoTwistToroidal} show that large families of Lorenz knots that are satellite knots satisfy \refconj{LorenzSatellite}. However, proving the full conjecture requires more than showing the conjecture for families. Every possible satellite knot that is a Lorenz knot must be shown to have companion and pattern as claimed. To that end, results such as \refcor{BWIntro} also make progress towards the conjecture, since they show that additional large families of Lorenz knots are hyperbolic, and therefore do not need to be considered for proving the conjecture. If we could completely classify all Lorenz knots that are hyperbolic or torus knots in terms of their description as a T-link, then proving the conjecture would require analysing the remaining T-links and their pieces in a torus decomposition.

In fact, a similar construction to that of \refthm{NoTwistToroidal} gives a partial converse to \refconj{LorenzSatellite}.

\begin{named}{\refcor{BWIntro}}
  For any two one-component T-links $K_1$ and $K_2$, there exists a satellite T-link $K$ such that after cutting $S^3-K$ along an essential torus, the components consist of the complement of $K_1$ in $S^3$, and the complement of $K_2$ in a solid torus.
\end{named}

Corollary~\ref{Cor:BWIntro} can be seen as a generalisation of a theorem of Birman and Williams \cite[Theorem~6.2]{BirmanWilliams}, who showed a similar result for Lorenz knot companions, and torus knot patterns. The proof is similar to theirs: we construct a link by arranging the pattern carefully within a solid torus neighbourhood of the companion, and prove that the result is still a T-link. These are the simplest examples of satellite T-links. Their construction gives evidence for the following. 

\begin{question}\label{Ques:NormQuestion}
  Let $K_1$ and $K_2$ be any Lorenz knots (i.e.\ one-component T-links). Is the satellite knot with pattern $K_1$ and companion $K_2$ always a Lorenz knot?
\end{question}

Observe that \refcor{BWIntro} gives evidence that the answer to Question~\ref{Ques:NormQuestion} is yes. However, it does not quite prove it, because we require $K_1$ to be twisted further before it fits into the position of a pattern for a T-link. This twisting is a homeomorphism of the solid torus, so does not affect the homeomorphism type of the knot in the solid torus. However, it does affect the satellite knot.


\subsection{Organisation}
In \refsec{Tangles} we introduce terminology and a few examples that will be used throughout the paper. Section~\ref{Sec:TwistingOnce} reviews an argument of Lee~\cite{Lee:Unknotted}, which finds an essential torus in a class of twisted torus knots that can be related to a simple case of T-links. This argument is extended to more general T-links in \refsec{TwistingMult}, allowing us to prove \refthm{SatelliteMultTwists}, which implies \refthm{SatelliteIntro}. We then give the proof of \refthm{SatelliteLorenzIntro}.

In \refsec{Characterising}, we prove \refthm{AugHyperbolic}, giving hyperbolic and satellite examples of augmented T-links.
This gives strong evidence for our amended \refconj{LorenzSatellite}.

In \refsec{NoTwistToroidal} we prove Theorem~\ref{Thm:NoTwistToroidal} and~\refcor{BWIntro}.

\section{Torus links and tangles}\label{Sec:Tangles}

In this section we set up results for links that will become the patterns in the satellites of the main theorem. Eventually they will be used to deal with the portion of the braid of a T-link containing the pairs $(a_1,b_1), \dots, (a_n,b_n)$. We keep the results as general as possible.

We first set up notation.

Throughout, let $D^2$ denote the unit disc in $\CC$, and $I=[-1,1]$ the unit disc in $\RR$. 

\begin{definition}\label{Def:TangleRStrands}
  A \emph{tangle on $r$ strands} is a compact 1-manifold $\tau$ embedded in a ball $B$ arranged as a solid cylinder $D^2\times [-1,1]$, with $\tau \cap (D^2\times\{1\})$ consisting of $r$ points, equally spaced on the axis $D^2\cap \RR$. Similarly $\tau \cap (D^2\times\{-1\})$ consists of $r$ points, equally spaced on $D^2\cap \RR$, and $\tau \cap (\bdy D^2\times I) = \emptyset$. An example for $r=3$ is shown on the left of \reffig{TangleBraid}.
\end{definition}

\begin{figure}
  \centering
  \includegraphics{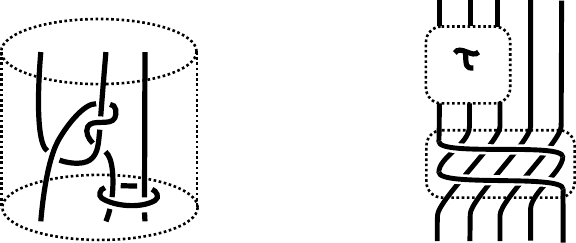}
  \caption{On the left is an example of a tangle on $r$ strands, for $r=3$. On the right is an example of a braid $\tau*(p,q)$, for $p=5, q=2$.}
  \label{Fig:TangleBraid}
\end{figure}

\begin{definition}\label{Def:TangleSum}
  Let $1<r\leq p$, $q>1$ be integers. Let $(p,q)$ denote the braid $(\sigma_1\dots \sigma_{p-1})^q$ on $p$ strands, so the torus link $T(p,q)$ is the closure of the braid $(p,q)$. Let $\tau_r$ be any tangle on $r$ strands. Beginning with the trivial braid on $p$ strands, replace a neighbourhood of the leftmost $r$ strands with $\tau_r$, followed by the braid $(p,q)$. Denote the result by $\tau_r*(p,q)$. An example is shown in \reffig{TangleBraid}, right. The braid closure of $\tau_r*(p,q)$ is a link, which we will denote by $\tau_r * T(p,q)$.
Because $\tau_r$ is a tangle and not necessarily a braid, the operation $*$ should be read left to right.
\end{definition}

\begin{lemma}\label{Lem:InvertTpq}
  Let $1<r\leq \min\{p,q\}$ be integers. Let $\tau_r$ be a tangle on $r$ strands in $D^2\times [-1,1]$. Let $\overline{\tau}_r$ denote the tangle on $r$ strands obtained from $\tau_r$ by rotating $180^\circ$ about the axis $(D^2\cap\RR)\times \{0\}\subset D^2\times [-1,1]$. Then the link $\tau_r * T(p,q)$ is equivalent to the link $\overline{\tau}_r * T(q,p)$.

  Moreover, let $F$ denote the Heegaard torus on which $T(p,q)$ is projected. Then the equivalence is by a small isotopy supported in a regular neighbourhood of $F$, followed by a homeomorphism of $S^3$ fixing $F$ and switching the two solid tori bounded by $F$. 
\end{lemma}

Informally, one sees the equivalence of \reflem{InvertTpq} by sliding $\tau_r$ from longitudinal strands to meridianal strands of $T(p,q)$, then moving one's head from one side of $F$ to the other in $S^3$, keeping everything else fixed. See \reffig{TangleSum}.

\begin{figure}
  \centering
\begingroup%
  \makeatletter%
  \providecommand\color[2][]{%
    \errmessage{(Inkscape) Color is used for the text in Inkscape, but the package 'color.sty' is not loaded}%
    \renewcommand\color[2][]{}%
  }%
  \providecommand\transparent[1]{%
    \errmessage{(Inkscape) Transparency is used (non-zero) for the text in Inkscape, but the package 'transparent.sty' is not loaded}%
    \renewcommand\transparent[1]{}%
  }%
  \providecommand\rotatebox[2]{#2}%
  \newcommand*\fsize{\dimexpr\f@size pt\relax}%
  \newcommand*\lineheight[1]{\fontsize{\fsize}{#1\fsize}\selectfont}%
  \ifx\svgwidth\undefined%
    \setlength{\unitlength}{250.80540848bp}%
    \ifx\svgscale\undefined%
      \relax%
    \else%
      \setlength{\unitlength}{\unitlength * \real{\svgscale}}%
    \fi%
  \else%
    \setlength{\unitlength}{\svgwidth}%
  \fi%
  \global\let\svgwidth\undefined%
  \global\let\svgscale\undefined%
  \makeatother%
  \begin{picture}(1,0.3177631)%
    \lineheight{1}%
    \setlength\tabcolsep{0pt}%
    \put(0,0){\includegraphics[width=\unitlength,page=1]{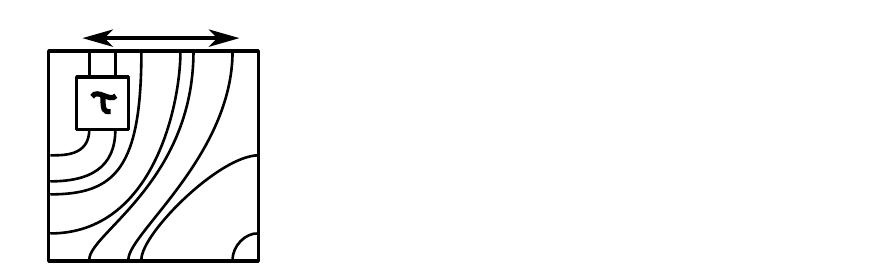}}%
    \put(0.16236754,0.28893081){\color[rgb]{0,0,0}\makebox(0,0)[lt]{\lineheight{0}\smash{\begin{tabular}[t]{l}$p$\end{tabular}}}}%
    \put(0,0){\includegraphics[width=\unitlength,page=2]{TangleTpq2.pdf}}%
    \put(-0.0021026,0.07960524){\color[rgb]{0,0,0}\makebox(0,0)[lt]{\lineheight{0}\smash{\begin{tabular}[t]{l}$q$\end{tabular}}}}%
    \put(0,0){\includegraphics[width=\unitlength,page=3]{TangleTpq2.pdf}}%
    \put(0.94219807,0.11093282){\color[rgb]{0,0,0}\makebox(0,0)[lt]{\lineheight{0}\smash{\begin{tabular}[t]{l}$p$\end{tabular}}}}%
    \put(0.79547076,0.28914448){\color[rgb]{0,0,0}\makebox(0,0)[lt]{\lineheight{0}\smash{\begin{tabular}[t]{l}$q$\end{tabular}}}}%
    \put(0,0){\includegraphics[width=\unitlength,page=4]{TangleTpq2.pdf}}%
  \end{picture}%
\endgroup%

  \caption{On the left is $\tau*T(p,q)$. Isotope the tangle $\tau$ in the direction of the arrow shown to move $\tau$ into a tangle in the meridianal direction. Then rotate $180^\circ$ in a diagonal axis to exchange the solid tori, and to exchange meridian and longitude. We obtain $\overline{\tau}_r*T(q,p)$ on the right.}
  \label{Fig:TangleSum}
\end{figure}

\begin{proof}[Proof of \reflem{InvertTpq}]
The meridian of one of the solid tori bounded by $F$ is the longitude of the other, and vice-versa. The standard homeomorphism taking the closed braid corresponding to $T(p,q)$ to the closed braid corresponding to $T(q,p)$ exchanges meridian and longitude and switches the two solid tori bounded by $F$. We need to consider how this affects the tangle $\tau_r$.

The tangle begins on the leftmost $r$ strands of the $p$ strands running in the longitudinal direction on the closed $T(p,q)$. Because the tangle is on $r \leq \min\{p,q\}$ strands, we may isotope it slightly to move it onto the $q$ parallel strands running in a meridianal direction; see \reffig{TangleSum}, left.

The homeomorphism of $S^3$ that switches solid tori and switches the roles of meridian and longitude can be expressed by a rotation about a diagonal of the square gluing to the torus $F$, as in \reffig{TangleSum}, middle. Then when we perform the homeomorphism (moving the head), the braid is now on the $q$ strands in the longitudinal direction of $T(q,p)$, on the left, but rotated vertically from the original position in $T(p,q)$; see \reffig{TangleSum}.
\end{proof}

The following are examples of tangles we will use in this paper.

\begin{example}[Braid associated with a T-link]
Let $1<r_1<\dots<r_k \leq r$, and $s_1, \dots, s_k$ be positive integers, and let $\tau_r = (r_1,s_1)*\dots*(r_k,s_k)$ be a braid of the form:
\[ (\sigma_1\dots \sigma_{r_1-1})^{s_1}\dots(\sigma_1\dots \sigma_{r_k-1})^{s_k} \]
Then $\tau_r * T(p,q)$ is the T-link $T((r_1,s_1),\dots(r_k,s_k),(p,q))$.

However, $\overline{\tau}_r*T(q,p)$ is not in the form of a T-link, for most general values of $s_1,\dots,s_k$. Rather, it is the braid
\[ (\sigma_{r_k-1}\sigma_{r_k-2}\dots\sigma_{1})^{s_k}\dots(\sigma_{r_1-1}\sigma_{r_1-2}\dots\sigma_1)^{s_1}(\sigma_1\sigma_2\dots\sigma_{q-1})^p. \]

The following lemma was observed by Birman and Kofman~\cite[Corollary~3]{BirmanKofman}.

\begin{lemma}\label{Lem:BirmanKofmanSym}
If each $s_i=t_ir_i$ is a multiple of $r_i$, then
\[ \tau_r = (r_1,s_1)*\dots*(r_k,s_k) = (\sigma_1\dots \sigma_{r_1-1})^{s_1}\dots(\sigma_1\dots \sigma_{r_k-1})^{s_k} \]
is equivalent to $\overline{\tau}_r$ in the braid group. Hence $\tau_r*T(q,p)=T((r_1,s_1), \dots, (r_k,s_k),(p,q))$ is equivalent to $\overline{\tau}_r*T(q,p)$ as closed braids.
\end{lemma}

\begin{proof}
The braid corresponding to $(r_i,s_i)$ is equivalent to $t_i$ full twists, which is invariant under rotation in $D^2\times\{0\}$. Moreover, full twists commute with other elements of the braid group (see, for example~\cite{Gonzalez-Meneses:BraidGroup}), 
so we may adjust them to be in the order required in the definition of a T-link. 
Then the T-link $\tau_r*T(p,q)=T((r_1,t_1r_1),\dots,(r_k,t_kr_k),(p,q))$ becomes the T-link $T((r_1,t_1r_1),\dots,(r_k,t_kr_k),(q,p))$ under the equivalence of \reflem{InvertTpq}.
\end{proof}
\end{example}

\begin{example}[Fully augmented T-link]
Again let $1<r_1<\dots<r_k<r$. Start with the trivial braid consisting of $r$ intervals of the form $* \times I$ in $D^2\times I$. For each $r_i$, let $J_{r_i}$ be an unknot in $D^2\times I$ bounding a level disc meeting the first $r_i$ strands of the trivial braid. Let $\tau_r$ be the union of the trivial braid and the unknots $J_{r_1},\dots,J_{r_k}$. Observe that again in this case, $\overline{\tau}_r*T(q,p)$ and $\tau_r*T(q,p)$ agree. 
\end{example}

\begin{lemma}\label{Lem:OmitR=Q}
In a T-link $T((r_1,r_1s_1), \dots, (r_m,r_ms_m),(p,q))$, we may always assume that the $r_i$ are neither equal to $p$ nor $q$, for all $i=1, \dots, m$.

More precisely, if $1<r_1<\dots<r_n=q < r_{n+1} < \dots < r_m<p$, and $s_1,\dots, s_m>0$, then the T-link
\[
  K_1 =  T((r_1,s_1),\dots,(r_{n-1},s_{n-1}),(q,qs_n),
  (r_{n+1},r_{n+1}s_{n+1}),\dots,(r_m,r_ms_m),(p,q))
\]
is equivalent to the T-link
\[
  K_2 =  T((r_1,s_1),\dots,(r_{n-1},s_{n-1}),
   (r_{n+1},r_{n-1}s_{n-1}),\dots, (p+qs_n, q)).
\]
And if $r_m=p$, the link
\[ T((r_1,s_1), \dots, (r_{n-1},s_{n-1}),(p,ps_m),(p,q))\]
is equivalent to the link
\[ T((r_1,s_1),\dots, (r_{n-1},s_{n-1})(p,ps_m+q)).\]
\end{lemma}

\begin{proof}
The final statement follows from the straightforward fact that the braid $(p,a)*(p,q)$ equals the braid $(p,a+q)$. Thus we focus on the case that $r_n=q$.
  
Let $\tau_r$ be the braid $(r_1,s_1)*\dots*(r_{n-1},s_{n-1})$ on $r_{n-1}$ strands. Let $C_n, C_{n+1}, \dots, C_m$ be disjoint unknots embedded in the complement of the link $\tau_r*T(p,q)$, encircling the first $r_n, \dots, r_m$ strands of the braid, respectively. Then the T-link $K_1$ is obtained from the link $(\tau_r*T(p,q))\cup C_n \cup \dots \cup C_m$ by performing $1/s_j$ Dehn filling on the link component $C_j$, for $j=n, \dots, m$.

The component $C_n$ encircles $r_n=q$ strands; isotope this to encircle the $q$ overstrands of the tangle $(p,q)$. Now apply the isotopy of \reffig{TangleSum}. This takes $C_n$ to an unknot encircling $q$ strands in a link $\overline{\tau}_r*T(q,p)$. The link components $C_{n+1}, \dots, C_m$ are taken to some link components $\overline{C}_{n+1}, \dots, \overline{C}_m$, but these are all disjoint from the disc bounded by $C_n$, so we ignore them for now. Perform the full twist given by $1/s_n$ Dehn filling on $C_n$, taking the link $\overline{\tau}_r*T(q,p)$ to the link $\overline{\tau}_r*T(q,p+qs_n)$. This does not affect any link components $\overline{C}_{n+1}, \dots, \overline{C}_n$, as they are disjoint from the disc bounded by $C_n$. Now undo the isotopy of \reffig{TangleSum}. The link $\overline{\tau}_r*T(q,p+qs_n) \cup \overline{C}_{n+1} \cup \dots \cup \overline{C}_m$ is taken to the link $\tau_r*T(p+qs_n,q)\cup C_{n+1}\cup \dots \cup C_m$. Perform the $1/s_j$ Dehn filling on $C_j$ for each $j=n+1, \dots, m$ to obtain the result.
\end{proof}

Note in the above proof, we used the hypothesis that in the T-link, there are only full twists on $r_j$ strands for all $r_j>q$: we isotoped a full twist on $q$ strands past these full twists onto the meridianal strands. When there is some $r_j>q$ and the T-link includes a braid $(r_j,s_j)$ for $s_j$ not a multiple of $r_j$, the above proof will not apply and the result is not necessarily true.

\section{Twisting once}\label{Sec:TwistingOnce}

In this section we will characterise when a certain link associated with a T-link of the form $T((r,rs),(p,q))$ is satellite. The work in this section is essentially due to Lee in~\cite{Lee:Cable}. We include this work for two reasons: first to set up notation, and second to present arguments that are easier to state in this simpler case before generalising in the next section. 

Let $1<r <q < p$ be integers, and let $K=\tau_r*T(p,q)$ be any link as in \refdef{TangleSum}. There are two natural surfaces of projection with which to describe $K$, as follows. One is the Heegaard torus on which $T(p,q)$ is projected; call this $F$. The second is the sphere $P$ forming the usual plane of projection of the braid. This sphere separates $S^3$ into two balls, one above the sphere and one below. Call the ball below the plane of projection $B$. Isotope the tangle $\tau_r$ to lie in the $q$ strands running in a meridianal direction over the $p$ strands of the braid, so that $B\cap F$ is an annulus meeting no crossings of the diagram. This annulus is shown for an example in \reffig{Annulus}.

\begin{figure}
  \centering
\begingroup%
  \makeatletter%
  \providecommand\color[2][]{%
    \errmessage{(Inkscape) Color is used for the text in Inkscape, but the package 'color.sty' is not loaded}%
    \renewcommand\color[2][]{}%
  }%
  \providecommand\transparent[1]{%
    \errmessage{(Inkscape) Transparency is used (non-zero) for the text in Inkscape, but the package 'transparent.sty' is not loaded}%
    \renewcommand\transparent[1]{}%
  }%
  \providecommand\rotatebox[2]{#2}%
  \newcommand*\fsize{\dimexpr\f@size pt\relax}%
  \newcommand*\lineheight[1]{\fontsize{\fsize}{#1\fsize}\selectfont}%
  \ifx\svgwidth\undefined%
    \setlength{\unitlength}{280.81863213bp}%
    \ifx\svgscale\undefined%
      \relax%
    \else%
      \setlength{\unitlength}{\unitlength * \real{\svgscale}}%
    \fi%
  \else%
    \setlength{\unitlength}{\svgwidth}%
  \fi%
  \global\let\svgwidth\undefined%
  \global\let\svgscale\undefined%
  \makeatother%
  \begin{picture}(1,0.43619523)%
    \lineheight{1}%
    \setlength\tabcolsep{0pt}%
    \put(0,0){\includegraphics[width=\unitlength,page=1]{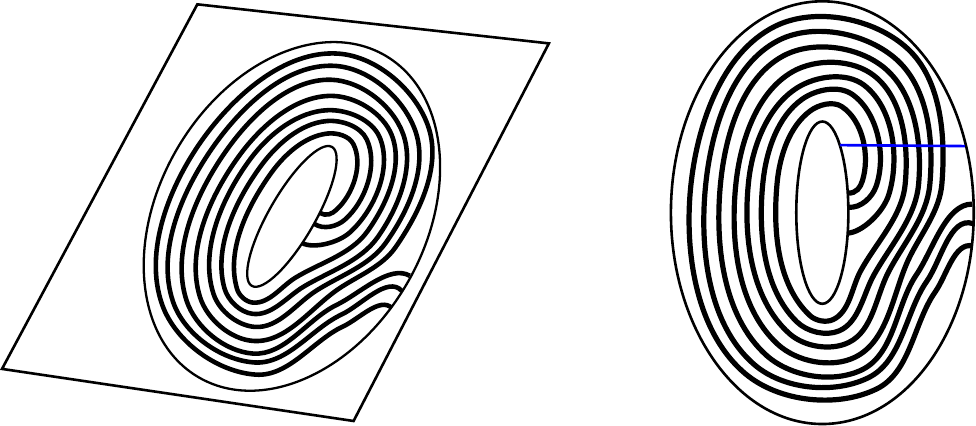}}%
    \put(0.45610326,0.09791326){\color[rgb]{0,0,0}\makebox(0,0)[lt]{\lineheight{1.25}\smash{\begin{tabular}[t]{l}$B$\end{tabular}}}}%
    \put(0,0){\includegraphics[width=\unitlength,page=2]{Annulus.pdf}}%
    \put(0.46945709,0.35354345){\color[rgb]{0,0,0}\makebox(0,0)[lt]{\lineheight{1.25}\smash{\begin{tabular}[t]{l}$P$\end{tabular}}}}%
    \put(0.04118119,0.10268255){\color[rgb]{0,0,0}\makebox(0,0)[lt]{\lineheight{1.25}\smash{\begin{tabular}[t]{l}$P\cap F$\end{tabular}}}}%
    \put(0,0){\includegraphics[width=\unitlength,page=3]{Annulus.pdf}}%
  \end{picture}%
\endgroup%

  \caption{Left: The link $\tau_r*T(p,q)$ shown lying above the projection plane $P$, so that $B\cap F$ is an annulus meeting no crossings on $F$. Right: The link meets $F\cap B$ in a collection of arcs disjoint from $\tau_r$. Shown is the case $p=7, q=3$. This is the ``back'' of the torus $F$, viewing the link with our heads on the outside of $B$ (from the front). The horizontal blue arc meets the link exactly $p$ times.}
  \label{Fig:Annulus}
\end{figure}

The intersection of $K$ with the annulus $B\cap F$ is a 1-manifold consisting of $q$ arcs with endpoints on either side of the annulus, wrapping around the core of the annulus in such a way that a trivial essential arc from one side of the annulus to the other, with endpoints just above the braid, meets exactly $p$ arcs. See \reffig{Annulus}.

\begin{definition}\label{Def:AugComponent}
Let $1<r<q<p$, and let $K=\tau_r*T(p,q)$ be as above. 
Let $a,b$ be integers with $0\leq a,b \leq p$ and $a+b<p$, and let $J_{a,b}$ be an unknot bounding a disc such that the interior of that disc meets $F$ transversely in a single arc intersecting exactly $b$ strands. Position $J_{a,b}$ to lie above the tangle $\tau_r$ and braid $(p,q)$ in the braid, so that it meets the annulus $F\cap B$ in exactly two points. Arrange these two points to lie on the blue horizontal arc shown in \reffig{Annulus}, right, so that the one on the left has exactly $a$ intersection points of $K$ to its left, and between $a$ and $b$ lie exactly $b$ intersection points. 

More generally, given $a_1, b_1, \dots, a_n, b_n$ satisfying $0\leq a_i, b_i<p$ and $a_i+b_i\leq p$ for all $i$, we may take unknots $J_{a_{1},b_{1}}, \dots, J_{a_n,b_n}$ as above, chosen so that the $(i+1)$-th is pushed slightly above the $i$-th, so that all are disjoint. See \reffig{UnlinksJ}.
\end{definition}

\begin{figure}
  \centering
  \includegraphics{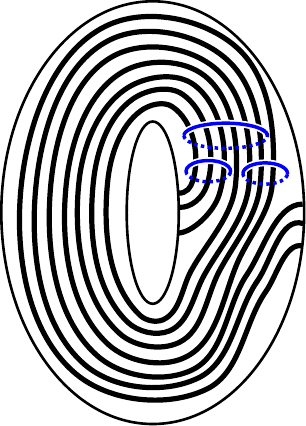}
  \caption{Shown are links $J_{0,6}$, $J_{0,3}$, and $J_{4,3}$}
  \label{Fig:UnlinksJ}
\end{figure}

\begin{proposition}\label{Prop:OneComponentJ}
  Let $1<r<q < p$ be integers. Let $b$ be an integer with $0<b<p$, and $\tau_r$ a tangle on $r$ strands. Let $K=\tau_r*T(p,q)$ and let $J_{0,b}$ be as in \refdef{AugComponent}. Suppose $b=sq$ for some integer $s>1$. Then $J_{0,b}$ may be isotoped in the complement of $K$ to be disjoint from $F$. The link $K\cup J_{0,b}$ is a satellite link, with an essential torus $T$ parallel to $F$, obtained by pushing $F$ slightly off of $K$. Within the solid torus bounded by $T$, the link component $J_{0,b}$ forms the torus knot $T(s,1)$. On the opposite side of $T$, the link $K$ has the form of a link $\overline{\tau}_r*T(q,p)$ within a solid torus.
\end{proposition}

\begin{proof}
The isotopy of $J_{0,b}$ to be disjoint $F$ occurs completely within the ball $B$ of \reffig{Annulus}. That is, we isotope $J_{0,b}$ only within the ball that lies below the plane of projection in \reffig{Annulus}. To make this easier to visualise, in \reffig{OneComponentJ} we have flipped the projection plane $P$ over, and we are looking from inside $B$. Observe that the portion of the diagram on $P\cap B$ in \reffig{OneComponentJ} has been rotated $180^\circ$ from \reffig{Annulus}; this is to indicate the change of position of our heads.
  
Within the ball $B$, $J_{0,b}\cap B$ forms a half circle with one endpoint just to the right of the inner boundary component of the annulus $F\cap B$, bounding $b=sq$ strands of $K$ on $F\cap B$, with the other endpoint between the $sq$-th and $(sq+1)$-th strands. See \reffig{OneComponentJ}, left.

\begin{figure}
  \centering
  \includegraphics{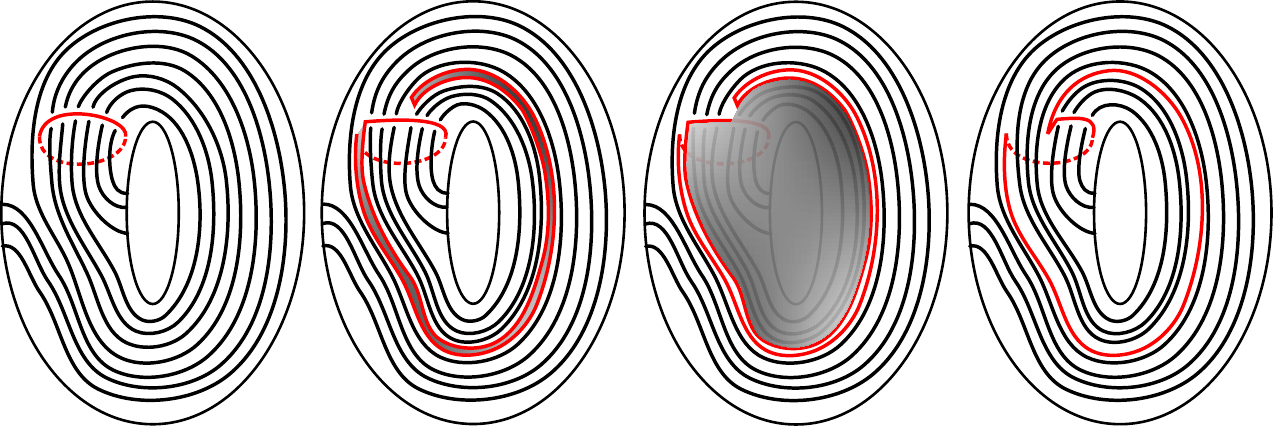}
  \caption{On the left, the arc of $J_{0,b}$ lying below $F\cap B$ is shown (with our head in the ball $B$). Isotope $J_{0,b}$, keeping all of $J_{0,b}$ fixed except a small arc, sliding this arc parallel to the annulus $F-K$ as shown in the next frame. After the arc has been pulled once around the core curve of $F\cap B$, a portion of the arc bounds a disc in $B$; this may be pulled tight, undoing a loop, as shown on the right. Then repeat.}
  \label{Fig:OneComponentJ}
\end{figure}

Keeping these two endpoints of $J_{0,b}\cap B$ fixed, isotope a small arc of $J_{0,b}$ through $B$. Take this arc to lie just above the intersection point of $J_{0,b}$ with $F$ lying between the $sq$-th and $(sq+1)$-th strands of $K$. Now isotope the arc, laying down arc in the annulus $F-K$ following the strands of $K$, with a portion of arc above the annulus, running parallel to the annulus and the arc. See \reffig{OneComponentJ}, middle.

After pulling this arc once around the core curve of $F\cap B$, a portion of the arc forms a loop bounding a disk with an arc running over $q$ strands. Undo this loop, as shown in the right of \reffig{OneComponentJ}. The arc of $J_{0,b}$ now consists of one arc embedded in $F-K$, running parallel to $K$ once around the core curve of $F\cap B$, and an arc in the interior of $B$ connecting the endpoint of the arc on $F-K$ to the original intersection point of $J_{0,b}$ on the inside of $F\cap B$. This arc in the interior of $B$ now forms a half circle bounding $(s-1)q$ strands. 

Repeat this process. Each time we isotope around $F\cap B$, we adjust the arc of $J_{0,b}$ in the interior of $B$ to be a half circle bounding $q$ fewer strands. After $s$ times around, the arc forms a $T(s,1)$ curve on the annulus $F\cap B$ disjoint from $K$, and the half circle in the interior of $B$ bounds no strands, hence can be isotoped into $F-K$.

Now push the curve $T(s,1)$ slightly off of $B$. The result is disjoint from $B$, and disjoint from $F$. Consider a torus $T$ slightly between $T(s,1)$ and $K$, parallel to $F$. To one side it bounds a solid torus $W$ containing the torus knot $T(s,1)$ isotopic to $J_{0,b}$, to the other it bounds a solid torus $V$ containing $K$. The solid torus $V$ has linking number $s$ with $T(s,1)$, hence $T$ is essential on the side containing $T(s,1)$. When viewed from the other side, the solid torus $W$ has linking number $q$ with $K$, which has the form $\overline{\tau}_r*T(q,p)$ within $V$, as in \reflem{InvertTpq}. Hence it can be neither compressible nor boundary parallel.
\end{proof}

\begin{figure}
  \centering
  \includegraphics{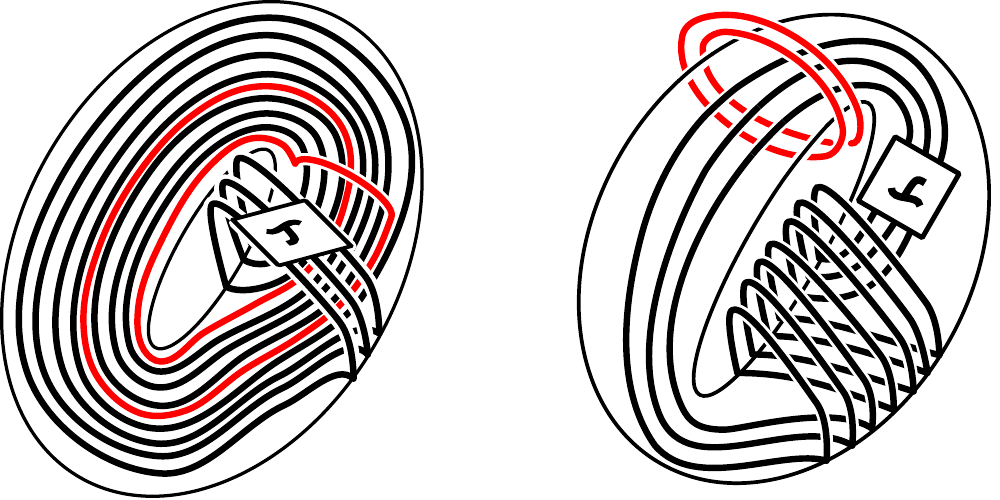}
  \caption{Proposition~\ref{Prop:OneComponentJ} is illustrated for $\tau*T(7,3)\cup J_{0,6}$ (left), which equals $\overline{\tau}*T(3,7)\cup J_{0,6}$ (right). The component $J_{0,6}$ is shown in red in each diagram. The essential torus $T$ is parallel to a Heegaard torus for $S^3$, and separates the two link components.}
  \label{Fig:EssentialTorus}
\end{figure}

Proposition~\ref{Prop:OneComponentJ} is illustrated in \reffig{EssentialTorus}. This shows two views of the identical links $\tau*T(7,3) \cup J_{0,6}$ and $\overline{\tau}*T(3,7)\cup J_{0,6}$. The link component $J_{0,6}$ has been isotoped to form a $(2,1)$-torus link on the left. The essential torus $T$ is parallel to a Heegaard torus and separates $\tau*T(7,3)$ from $T(2,1)$. On the right is the view with our heads moved to the other side of the torus knot. The component $\tau*T(7,3)$ has become $\overline{\tau}*T(3,7)$, and the component $J_{0,6}$ has become a $T(1,2)$ torus link. Again the essential torus is parallel to a Heegaard torus for $S^3$, separating the two link components.

\begin{theorem}\label{Thm:LeeWithConverse}
Let $p,q$ be relatively prime integers, with $1<q<p$, and let $a$ be an integer such that $1<a<p$. The link $T(p,q)\cup J_{0,a}$ is satellite if and only if $a=sq$ is a multiple of $q$. If $a=sq$ is a multiple of $q$, $J_{0,a}$ is isotopic to a knot $T(s,1)$ disjoint from the Heegaard torus $F$ on which $T(p,q)$ is projected, and there is an essential torus obtained by isotoping $F$ slightly off of $T(p,q)$. 
\end{theorem}

\begin{proof}
When $a=sq$, the fact that the link is satellite with essential torus as described follows immediately from \refprop{OneComponentJ}. 

When $a\neq sq$, then $T(p,q)\cup J_{0,a}$ is hyperbolic by work of Lee~\cite[Proposition~5.7]{Lee:Unknotted}, thus atoroidal. 
\end{proof}

\section{Twisting multiple times}\label{Sec:TwistingMult}

In this section, we extend \refprop{OneComponentJ} to multiple components, isotoping them simultaneously to be disjoint from
the projection torus of $T(p,q)$ under certain conditions.

\begin{lemma}\label{Lem:MultipleJ}
Let $1<r<q<p$ be integers, $\tau_r$ a tangle on $r$ strands, and $K=\tau_r*T(p,q)$ a link as in \refdef{TangleSum}. 
Let $(a_1,b_1), \dots, (a_n,b_n)$ be pairs of integers such that each $a_i$ is a multiple of $q$ (possibly $a_i=0$), and each $b_i$ is a non-zero multiple of $q$, and such that $a_i+b_i<p$, for $i=1, \dots, n$. Let $J_{a_i,b_i}$ be an unknotted link component as in \refdef{AugComponent}, so each $J_{a_i,b_i}$ encircles a multiple of $q$ strands, and has a multiple of $q$ strands lying to the inside of its innermost endpoint on $F\cap B$. Then:
\begin{enumerate}
\item All link components $J_{a_i,b_i}$ may be isotoped simultaneously in the complement of $K$ to be disjoint from the projection torus $F$ of $T(p,q)$.
\item The link $K\cup\left(\bigcup_{i=1}^n J_{a_i,b_i}\right)$ is a satellite link, with an essential torus $T$ parallel to $F$.
\item The torus $T$ bounds a solid torus on one side containing all link components $J_{a_i,b_i}$, all forming torus unknots of the form $T(s_i,1)$ where $b_i=s_iq$. On the other side, it bounds a solid torus containing $K$, having the form $\overline{\tau}_r*T(q,p)$.
\end{enumerate}
\end{lemma}

\begin{proof}
The proof is nearly identical to that of \refprop{OneComponentJ}, only we keep track of multiple link components simultaneously. Again the isotopy happens only in the ball $B$ below the plane of projection, where the unknots $J_{a_i,b_i}$ form half circles, each encircling a multiple of $q$ strands. Starting with the region of $F\cap B-K$ containing the outermost endpoint(s) of the half circles, isotope all half circles meeting that region to lay down an arc on $F-K$, sliding parallel to $F-K$, keeping the largest circles on the outside. See \reffig{MultipleJ}, left two panels. (Again in \reffig{MultipleJ} we have rotated the picture to put our heads inside the ball $B$.)

\begin{figure}
  \centering
  \includegraphics{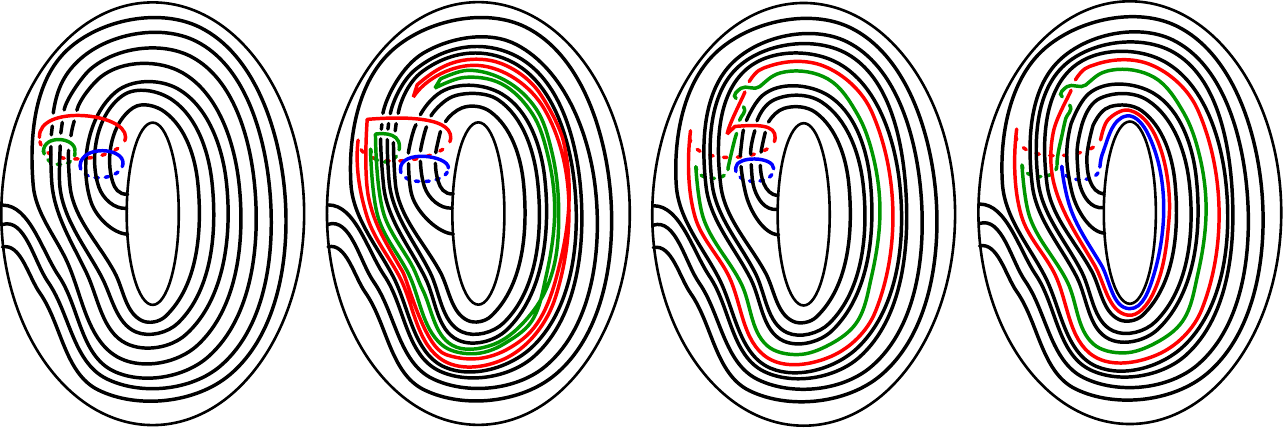}
  \caption{When there are multiple link components, as shown on the left, starting with the outermost points of intersection, isotope just as in \reffig{OneComponentJ}. After one pass around the core of $F\cap B$, when pulling loops tight they must form a full twist, as shown second to right. The right panel shows the result of isotoping all components for this example.}
  \label{Fig:MultipleJ}
\end{figure}

After one pass around the annulus $F\cap B-K$, a portion of the arcs that have been pulled around the annulus now form loops. These may be `unlooped,' provided we add a full twist to all loops involved as on the second to right of \reffig{MultipleJ}. This replaces the half circle of such $J_{a_i,b_i}$ with an arc on $F\cap B-K$, embedded aside from where it forms crossings coming from full twists about other $J_{a_j,b_j}$, and a half circle bounding $q$ fewer strands of $K$. Observe that adding the full twist preserves the fact that the $J_{a_j,b_j}$ have linking number zero with each other.

Repeat, moving from the outside to the inside, isotoping all components $J_{a_i,b_i}$ to intersect $B$ only in the neighbourhood of the annulus $F\cap B-K$ to the outside, forming a half-circle bounding $q$ fewer strands of $K$ after each pass. Because the $a_i$ and $b_i$ are all multiples of $q$, for each $J_{a_i,b_i}$ the portion of the half circle obtained above eventually bounds zero strands, and can be isotoped into $F$.

As before, after isotopy the $J_{a_i,b_i}$ form knots $T(s_i,1)$. When all are embedded on $F$, we may push past $F$ slightly to obtain the result.
\end{proof}

\begin{lemma}\label{Lem:InvertC}
  Let $J_1 = T(s_1,1), \dots, J_k= T(s_k,1)$ be torus (un)knots lying in an unknotted solid torus in $S^3$ obtained from \reflem{MultipleJ}, where $J_i=J_{a_i,b_i}$ with $a_1\leq a_i$ and $b_i\leq b_1$ for all $i=2, \dots, k$.  Let $C$ be the core of the unknotted solid torus in $S^3$ such that $S^3-C$ is the solid torus containing the $J_i$. Then there is an ambient isotopy of $S^3$ that takes $C$ to a torus knot $T(s_1,1)$, and takes each $J_i$ to an unknot bounding a disc meeting $T(s_1,1)=C$ in $s_i$ points, for $i=1, \dots, k$. Moreover, the discs bounded by the $J_i$ are mutually disjoint.
\end{lemma}

\begin{proof}
The isotopy can be viewed as follows. First, each $T(s_i,1)$ can be isotoped to have an arc inside of the Heegaard torus $F$, and an arc on $F$ running $s_i$ times around the longitude. Shrink $C$ slightly and isotope it to meet $F\cap B$ in exactly two points, one near the innermost boundary of the annulus $F\cap B$ and one near the outermost boundary component. Now beginning with the outermost point of intersection, drag it around $F\cap B$ keeping it disjoint from the arcs of the $T(s_i,1)$, just as in the proof of \refprop{OneComponentJ}. After $s_1$ passes, an arc of $C$ lies outside of $B$ and an arc is embedded in $F\cap B$, and these two arcs form a torus knot $T(s_1,1)$; see \reffig{InvertC1}, left and second left.

\begin{figure}
  \centering
  \includegraphics{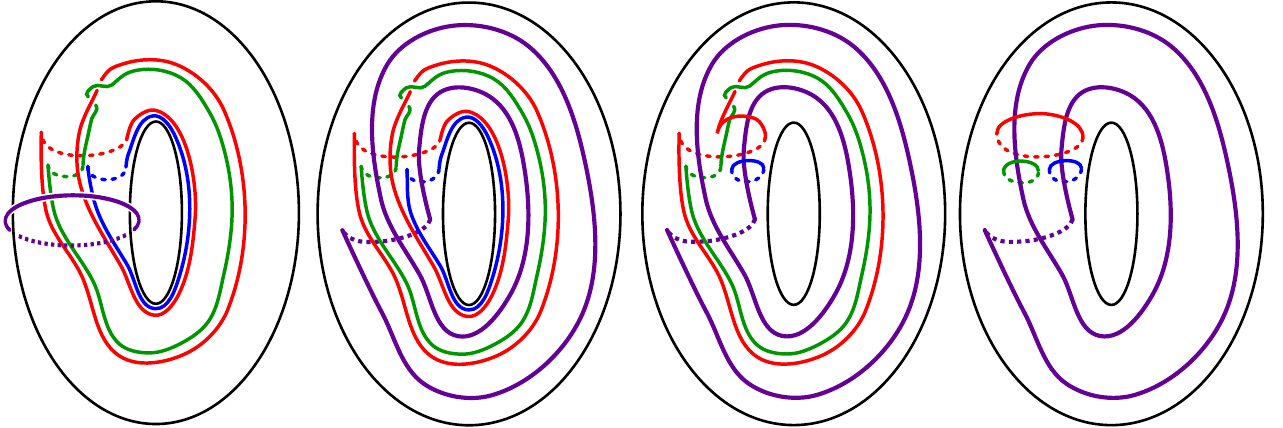}
  \caption{Left to right: The unknot $C$ is shown encircling the torus links $T(s_i,1)$. Isotope $C$ as in \reffig{OneComponentJ}, to obtain a torus knot $C=T(s,1)$. Now reverse the procedure of \reffig{MultipleJ}, unwinding components from the innermost side out. When finished, link components are as claimed.}
  \label{Fig:InvertC1}
\end{figure}

Reverse the process of the proof of \reflem{MultipleJ}, unwinding the knots $J_1=T(s_1,1)$ through $J_k=T(s_k,1)$ to form unknots with no crossings, each now encircling $C=T(s_1,1)$. Two steps of this process are shown for an example in the two rightmost panels of \reffig{InvertC1}.

Because each of the knots $J_1, \dots, J_k$ is unknotted and unlinked from the others, the components separate into unknotted, unlinked components. Because the linking number of $C$ and $J_j=T(s_j,1)$ is $s_j$ before isotopy, for $j=1, \dots, k$, the result of unwinding $J_j$ is an unknot with no crossings encircling $s_j$ strands of $C$ in its form $T(s_1,1)$. 
\end{proof}

\begin{theorem}\label{Thm:SatelliteMultTwists}
Let $a_1, b_1, \dots, a_n, b_n, s_1, t_1, \dots, s_m, t_m, p,q$ be integers satisfying:
\[ 1<a_1<\dots<a_n<q < qs_1<\dots<qs_m, \mbox{ and }
b_i, t_i >0 \mbox{ for all $i$.}
\]
Then the T-link
\[ T((a_1,b_1), \dots, (a_n,b_n),(s_1q,s_1qt_1), \dots, (s_mq, s_mqt_m),(p,q)) \]
is a satellite link with companion the T-link
\[ T((s_1,t_1s_1), \dots, (s_m,t_ms_m+1)) \]
and pattern the link $\overline{\tau}_r*T(q,p+s_1^2qt_1+\dots+s_m^2qt_m)$, where $\tau_r$ is the tangle containing the braid $(a_1,b_1)*\dots*(a_n,b_n)$. 
\end{theorem}

\begin{remark}
Observe that when the solid torus containing the pattern
\[ \overline{\tau}_r*T(q,p+s_1^2qt_1+\dots+s_m^2qt_m)\]
is embedded as an unknotted solid torus in $S^3$, the pattern as a link in $S^3$ is equivalent to the link in $S^3$
\[ \tau_r*T(p+s_1^2qt_1+\dots+s_m^2qt_m,q) = T((a_1,b_1),\dots,(a_n,b_n),(p+s_1^2qt_1+\dots+s_m^2qt_m,q)), \]
which is a T-link. Thus these links satisfy \refconj{LorenzSatellite}.
\end{remark}

\begin{proof}[Proof of \refthm{SatelliteMultTwists}]
With $\tau_r$ as described, consider the link $\tau_r*T(p,q) \cup J_{0,s_1q}\cup \dots \cup J_{0,s_mq}$. The T-link in the theorem is obtained by performing $1/t_i$ Dehn filling on the link component $J_{0,s_iq}$, for $i=1, \dots, m$.

By \reflem{MultipleJ}, we may isotope the $J_{0,s_iq}$ simultaneously to form knots $T(s_i,1)$ inside the Heegaard torus $F$. Let $T$ be the essential torus obtained by pushing off $F$ slightly, so that $T$ bounds a solid torus containing $\overline{\tau}_r*T(q,p)$ on one side; call this solid torus $V$. Then the core of $V$ is an unknot encircling each of the $T(s_i,1)$. By \reflem{InvertC}, we may isotope $V$ to the torus knot $T(s_m,1)$, isotoping each $J_{0,s_iq}$ to a disjoint collection of unknots, bounding a disc meeting the first $s_i$ strands of $T(s_m,1)$, for $i=1, \dots, m$.

Now perform $1/t_i$ Dehn surgery on $J_{0,s_iq}$, for $i=1,\dots,m$. This replaces $V$ with the T-link
\[
T((s_1,t_1s_1), \dots, (s_{m-1}, t_{m-1}s_{m-1}), (s_m,t_ms_m+1)).
\]

The twisting also affects the link $\overline{\tau}_r*T(q,p)$ inside $V$. We may arrange the link such that each time $V$ runs through one of the discs $D_i$ bounded by the link $J_{0,s_iq}$, exactly $q$ strands within $V$ run through $D_i$. Then performing $1/t_i$ Dehn surgery on each $J_{0,s_iq}$ adjusts the $T(q,p)$ torus knot to become the $T(q,p+s_1^2qt_1+\dots+s_m^2qt_m)$ torus knot; see, for example, \cite[p.~267]{Rolfsen}. Hence the pattern is the link $\overline{\tau}_r*T(q,p+s_1^2qt_1+\dots+s_m^2qt_m)$. 
\end{proof}

Theorem~\ref{Thm:SatelliteIntro} is an immediate consequence of \refthm{SatelliteMultTwists}.

\begin{proof}[Proof of \refthm{SatelliteIntro}]
Let $\tau_r$ be the braid consisting of full-twists $(a_1,a_1b_1)*\dots*(a_n,a_nb_n)$. Then $\overline{\tau}_r = \tau_r$, and
\[ \overline{\tau}_r*T(q,p+s_1^qt_1+\dots+s_m^2qt_m)=T((a_1, a_1b_1),\dots,(a_n,a_nb_n),(q,p+s_1^qt_1+\dots+s_m^2qt_m)), \]
as in \reflem{BirmanKofmanSym}. Theorem~\ref{Thm:SatelliteIntro} then follows from \refthm{SatelliteMultTwists} using this tangle $\tau_r$.
\end{proof}

When $\tau_r$ is nontrivial, the pattern is not a torus knot. Thus \refthm{SatelliteMultTwists} gives counterexamples to the conjecture attributed to Morton, provided such links do not have an alternate description as a cable on a T-link. We prove this is not the case for infinitely many examples using a theorem that follows immediately from the following result of Lee~\cite[Corollary~1.2]{Lee:2018Satellite}.

\begin{theorem}[Lee]\label{Thm:Lee}
Suppose $p,q$ are positive, relatively prime integers, and $r,s$ are positive integers such that $1<r<q<p$ and $s>1$. Then the twisted torus knot $T((r,rs),(p,q))$ is hyperbolic. 
\end{theorem}

\begin{lemma}\label{Lem:HyperbolicPandC}
There exist infinitely many choices of integers in \refthm{SatelliteMultTwists} such that the complement of the companion in $S^3$ and the complement of the pattern in $S^1\times D^2$ are both hyperbolic.

In particular, for any relatively prime integers $p$, $q$, and integers $a_1$, $c$, $s_1$, $s_2$, $t_1$, and $t_2$ satisfying $1< a_1 < s_1 q < s_2 q <p$ and $c, t_1, t_2 >1$, the T-link
\[ T((a_1,a_1c), (s_1 q, s_1 q t_1), (s_2 q, s_2 q t_2), (p,q)) \]
is a satellite link whose companion is a hyperbolic twisted torus knot, and whose pattern is a hyperbolic twisted torus knot in a solid torus.
\end{lemma}

\begin{proof}
Take integers as in the statement of the lemma. 
Observe that $s_1 < s_2t_2+1$, so $s_1$ cannot be a multiple of $s_2t_2+1$.
Also observe that $a_1<p+s_1^2qt_1+s_2^2qt_2$, so $a_1$ is not a multiple of $p+s_1^2qt_1+s_2^2qt_2$.

By \refthm{SatelliteMultTwists}, the T-link
\[T((a_1, a_1c),(s_1q,s_1qt_1),(s_2q,s_2qt_2),(p,q))
\]
is a satellite link with companion the T-link
\[ T((s_1,t_1s_1),(s_2,t_2s_2+1)) 
\]
and pattern the link $\overline{\tau}_r*T(q,p+s_1^2qt_1+s_2^2qt_2)$, where $\tau_r$ is the tangle containing the braid $(\sigma_1\dots\sigma_{a_1-1})^{a_1c}$. Because this is a full twist, $\overline{\tau}_r$ equals $\tau_r$. This forms the T-link $T((a_1,a_1c),(q,p+s_1^2qt_1+s_2^2qt_2))$.

By symmetry,
\[ T((a_1,a_1c),(q,p+s_1^2qt_1+s_2^2qt_2)) = T((a_1,a_1c),(p+s_1^2qt_1+s_2^2qt_2, q)). \]
Therefore by \refthm{Lee}, both the companion and the pattern have hyperbolic complement, when viewed as links in $S^3$.

We actually wish to show something slightly different, namely that the complement of the pattern in a solid torus is hyperbolic. Let $J$ be the braid axis for the pattern $P=T((a_1,a_1c),(q,p+s_1^2qt_1+s_2^2qt_2))$. We will show $S^3-(P\cup J)$ is hyperbolic.

To do so, we apply work of Ito, particularly \cite[Example~5.7]{Ito:BraidOrdering}. Here, it is shown that if $p+s_1^2qt_1+s_2^2qt_2$ is at least $3q$, then the geometric structure on $S^3-(P\cup J)$ agrees with the geometric structure on $S^3-P$. Since Lee has shown that $S^3-P$ is hyperbolic, this implies that the braid is pseudo-anosov, implying that $S^3-(P\cup J)$ is hyperbolic. 
\end{proof}

Lemma~\ref{Lem:HyperbolicPandC} is somewhat unwieldy as stated, requiring many choices of integers. In fact, by fixing values of all integers in the statement except $c$, and letting $c$ vary among integers greater than $1$, we obtain an infinite family of T-links whose companion and pattern in $S^1\times D^2$ are hyperbolic. In particular, choose $a_1=2$, $q=3$, $s_1=2$, $s_2=3$, $p=11$, and $t_1=t_2=2$. Then $s_2t_2+1=7$, and $p+s_1^2qt_1+s_2^2qt_2= 11+4*3*2+9*3*2=89$. Thus the T-link
\[ T((2,2c),(6,12),(9,18),(11,3)) \]
is satellite with companion the T-link $T((2,4),(3,7))$, and pattern the T-link
\[ P=T((2,2c),(3,89)) = T((2,2c),(89,3)), \]
both of which are hyperbolic. Let $J$ denote the braid axis for $P$. Since $89>9$, \cite[Example~5.7]{Ito:BraidOrdering} implies that the geometric structure on $S^3-(P\cup J)$ agrees with that on $S^3-P$, hence it is also hyperbolic.

Theorem~\ref{Thm:SatelliteLorenzIntro} from the introduction now is an immediate consequence. 

\begin{named}{\refthm{SatelliteLorenzIntro}}
There exist infinitely many Lorenz knots that are satellites for which there are exactly two components of the JSJ decomposition, and both components are hyperbolic. Thus by the uniqueness of the JSJ decomposition of a knot complement, the knots cannot be cables on Lorenz knots. 
\end{named}

\begin{proof}
  Take the Lorenz knots to be a sequence as in \reflem{HyperbolicPandC}. 
\end{proof}

For example, take the fixed choices of integers $a_1=2$, $q=3$, $s_1=2$, $s_2=3$, $p=11$, $t_1=t_2=a_1=2$ as above, and let $c$ be an arbitrary integer greater than $1$. Then $T((2,2c),(6,12),(9,18), (11,3))$ gives an infinite family required by \refthm{SatelliteLorenzIntro}.

\section{Hyperbolic links and satellites}\label{Sec:Characterising}

We have found infinitely many satellite Lorenz links. We would like to characterise all satellite Lorenz links. As in the previous section, infinitely many Lorenz links are obtained by starting with a link of the form $T(p,q)\cup J_{0,a_1}\cup \dots \cup J_{0,a_n}$, and performing $1/b_i$ Dehn filling on each component $J_{0,a_i}$, for $b_i>0$ and $i=1,\dots, n$. Such links form a family of T-links with full twisting. In this section, we show that if we require full twisting, then provided the amount of twisting $b_i$ is high, the only way to obtain satellites is by full twists on components $J_{0,a_i}$ with some $a_i$ a multiple of $q$, giving further evidence for \refconj{LorenzSatellite}. This is the content \refcor{FullTwistConj}, which is an immediate consequence of \refthm{AugHyperbolic}. In this section, we complete the proof of these results.

\begin{lemma}\label{Lem:Irreducible}
  Let $1<q<p$, and let $1<a_1<\dots<a_n<q$ be integers. 
  The complement of the link $T(p,q)\cup J_{0,a_1}\cup \dots \cup J_{0,a_n}$ is irreducible and boundary irreducible. 
\end{lemma}

\begin{proof}
If there exists a 2-sphere embedded in the link complement that does not bound a ball, then it contains $T(p,q)$ on one side and some $J_{0,a_i}$ on the other. But $J_{0,a_i}$ is unknotted, and thus the linking number of $J_{0,a_i}$ with $T(p,q)$ must be zero; this is a contradiction.

Similarly, suppose there exists a boundary compressing disc for the link complement. It cannot have boundary on $T(p,q)$, because torus knots are nontrivial. On the other hand, if its boundary lies on $J_{0,a_i}$ then again the linking number of $J_{0,a_i}$ and $T(p,q)$ is zero, a contradiction. 
\end{proof}

The proof of \refthm{AugHyperbolic} requires the following technical sublemma, which will be used to rule out essential tori in the link complement.

\begin{lemma}\label{Lem:ComponentsInBall}
Let $1<q<p$ where $p,q$ are relatively prime integers. Suppose $a_1, \dots, a_m$ are distinct integers such that no $a_i$ is a multiple of $q$. Consider the link
\[ L=T(p,q)\cup J_1 \cup \dots \cup J_m, \]
where we write $J_{0,a_i}$ as $J_i$.
Then there is no solid torus $V$ in $S^3$ with each $T(p,q)\cup J_i$ lying in a ball inside of $V$, and $\bdy V =T$ an essential torus in $S^3-L$. 
\end{lemma}

\begin{proof}
Suppose not. That is, suppose there exists a solid torus $V$ with $\bdy V=T$ an essential torus, and each $T(p,q)\cup J_i$ lies in a ball inside of $V$.
  
Let $F$ denote the Heegaard torus for $S^3$ on which $T(p,q)$ lies. Let $P$ be the surface $P=F-N(L)$. So $P$ is a sphere with $2m$ boundary components, two of which correspond to $\bdy N(T(p,q))$ and are isotopic to parallel nontrivial curves in $\bdy N(T(p,q))$, and two of which correspond to meridians of $N(J_i)$ for each $i$. The torus $T$ must intersect $P$, else it is embedded within one of the two handlebodies of $S^3-N(J_1\cup\dots\cup J_m)-N(F)$, contradicting the fact that it is essential. Thus $T$ intersects $P$ in some number of simple closed curves. Take $T$ to be a torus satisfying the hypotheses of the lemma such that the number of intersections of $T\cap P$ is minimal over all such tori. 

Because $T$ is incompressible, we may assume that $T\cap P$ does not consist of any closed curve bounding a disc in $P$: any innermost such curve also bounds a disc in $T$, and by irreducibility (\reflem{Irreducible}),
the union of the disc on $P$ and that on $T$ bounds a ball in $S^3-L$ that can be used to isotope $T$ through $P$, removing the intersection. Thus $T\cap P$ consists of closed curves encircling boundary components of $P$.

We now show that there is no curve of $T\cap P$ that is parallel to a boundary component of $P$ corresponding to $T(p,q)$. For if a curve of $T\cap P$ is parallel to $N(T(p,q))\cap P$, then we may isotope $T(p,q)$ through an annulus on $P$ with one boundary component on $N(T(p,q))$ and the other on $T\cap P$ to lie on $T$. Because $T(p,q)$ does not bound a disc in $S^3$, that annulus on $P$ does not have its second boundary component a meridian of $V$. Thus the wrapping number of $T(p,q)$ on $T=\bdy V$ after this isotopy is at least $1$. But $T(p,q)$ lies in a ball inside $T$, so this is impossible.

Similarly, there is no innermost curve of $T\cap P$ that encloses exactly one boundary component of $P$ corresponding to $T(p,q)$, as follows. By the above paragraph, any such curve must also enclose a boundary component corresponding to $J_i$ for some $i$. But then fill in all $J_i$; that is, perform trivial Dehn filling on each $J_i$. Because the curve of intersection is innermost, again there will be an annulus between $T$ and $T(p,q)$ in $S^3-T(p,q)$ that does not bound a disc in $S^3$. Again we may use the annulus to isotope $T(p,q)$ onto $T$, and $T(p,q)$ will have wrapping number at least $1$, contradicting the fact that it lies in a ball.

Next, we show there is no curve of $T\cap P$ that is parallel to a boundary component of $P$ corresponding to $J_i$ for some $i$. For if so, then this curve bounds an annulus on $P$ meeting one of the components $J_i$ in a meridian, hence the curve bounds a disc in $S^3$. Thus $J_i$ meets a meridian of the solid torus $V$ in a single point. But this is impossible: because $J_i$ lies inside a ball in $V$, it must intersect any meridianal disc an even number of times.

Similarly, no innermost curve of $T\cap P$ encloses exactly one boundary component of $J_i$ and one or two boundary components of other link components $J_j$ with $j\neq i$. For if not, fill in all such $J_j$, $j\neq i$. Again we obtain a curve on $T\cap P$ bounding an annulus disjoint from $J_i$, with its other boundary component a meridian of $N(J_i)$. So as above $J_i$ meets a meridian of the solid torus $V$ in a single point, which is a contradiction.

We conclude that each innermost curve $\gamma$ of $T\cap P$, bounding a punctured disc in $P$ disjoint from $T$, must meet boundary components of $P$ in pairs: if it meets one corresponding to $J_i$ or $T(p,q)$, then it must meet the other. It follows that there is some curve $\eta$ of $T\cap P$ bounding a punctured disc of $P$ disjoint from $T$, and also disjoint from both boundary components corresponding to $T(p,q)$. Then each of the boundary components on $P$ in the punctured disc correspond to meridians of the $J_i$, and thus $\eta$ bounds a disc in $S^3$. It follows that $\eta$ is a meridian of $V$. Because each curve of $T\cap P$ is parallel on $T=\bdy V$, each curve of $T\cap P$ bounds a meridian of $V$. 

Now let $\gamma$ be an innermost curve of $T\cap P$, bounding a subsurface of $P$ that is disjoint from $T$. 
Suppose first that another curve $\zeta$ of $T\cap P$ is parallel to $\gamma$ on $P$, so that $\gamma$ and $\eta$ bound an annulus $A$ on $P$ with interior disjoint from $T$. Then $\gamma$ and $\zeta$ are disjoint meridians of $V$. Construct a new torus $T'$ by replacing an annulus of $T$ between $\gamma$ and $\zeta$ with $A$. Note that one side of $T'$ still contains $V$, with a meridian of $\bdy V$ contained in the annulus of $T$ in $T'$. Thus the side of $T'$ containing $V$ bounds a solid torus $W$ in $S^3$, and $T(p,q)\cup J_i$ lie in balls in $W$ for all $i$. 

If $T'$ is compressible to the outside of $W$, then it must be unknotted in $S^3$. The curve bounding a compressing disc $D$ on $T'$ can be isotoped in $S^3-T'$ to meet the meridians $\gamma$ and $\zeta$ of $V$ exactly once each. Thus $\bdy D$ consists of an arc on $T$ and an arc on $P$. A small neighbourhood of $D$ is a ball in $S^3-(T(p,q)\cup J_{0,a}\cup J_{0,b})$. Isotope the arc of $\bdy D$ in $T$ through this ball to $P$ and slightly past. This reduces the number of components of intersection $T\cap P$, contradicting minimality of $T\cap P$. Thus we may assume that $T'$ is incompressible to the outside of $W$. Then $T'$ satisfies the hypotheses of this lemma. But observe that by isotoping $T'$ slightly off of the annulus $A\subset P$, the torus $T'$ is a torus satisfying the hypotheses of the lemma, but meeting $P$ fewer times. This contradicts minimality of $T\cap P$.

It follows that an innermost curve $\gamma_1$ of $T\cap P$ encircles both components of any $J_i$, and/or both components of $T(p,q)$, and there is no curve of $T\cap P$ parallel to $\gamma_1$ on $P$.

Now each component of $T\cap P$ is a nontrivial curve on $T$, and these are all disjoint. Now consider an innermost component $\eta$ of $T\cap P$ with respect to the boundary components corresponding to $T(p,q)$. That is, $\eta$ bounds a punctured subdisc of $P$ with two boundary components of lying on $\bdy N(T(p,q))$. This subdisc is isotopic into a subset of a meridian of $V$ in $S^3$. On the other side of $\eta$ on $P$, it bounds another component $P'$ of $P-V$. The surface $P'$ is a planar surface. Each of its boundary components is a meridian of $V$; hence tubes of $V$ connect boundary components of $P'$ on one side. It follows that there are an even number of boundary components of $P'$. 

Consider such a tube running from the meridian of $V$ meeting $T(p,q)$, call it $\gamma_1$, to another boundary component $\gamma_2$ of $P'$. The boundary of this tube consists of an annulus on $\bdy V=T$ with boundary $\gamma_1$ and $\gamma_2$, and two punctured discs on $P$, one meeting $T(p,q)$ twice, and one meeting some other link component(s) of $L$; both punctured discs bound discs in $V$. Perform trivial Dehn filling on all the $J_i$. This turns the boundary of the tube into an annulus in $S^3-T(p,q)$ with boundary components parallel to the boundary components of the essential annulus $F-N(T(p,q))$. Such an annulus is essential, and must be isotopic to $A=F-N(T(p,q))$.
But by a lemma of Lee \cite[Lemma~5.1]{Lee:Unknotted}, the tube must meet each component of $J_i$. But then $\gamma_1$ and $\gamma_2$ together bound discs containing all boundary components of $P$, and it follows that $\gamma_1$ and $\gamma_2$ are parallel in $P'$, bounding an annulus in $P'$, contradicting the work above. 
\end{proof}

\begin{lemma}\label{Lem:TwoJ}
Suppose $1<q<p$ and $p,q$ are relatively prime integers, and $a,b$ are distinct integers that are not multiples of $q$. Then
the link complement  $S^3-(T(p,q)\cup J_{0,a}\cup J_{0,b})$
is atoroidal.
\end{lemma}

\begin{proof}
Suppose not. Suppose that $S^3-T(p,q)\cup J_{0,b}\cup J_{0,a}$ admits an essential torus $T$.

Since $b$ is not a multiple of $q$, by \refthm{LeeWithConverse} the inclusion of $T$ into $S^3-(T(p,q)\cup J_{0,b})$ is inessential. It follows that the inclusion of $T$ into $S^3-(T(p,q)\cup J_{0,b})$ is either compressible or boundary parallel. Similarly, the projection of $T$ to $S^3-(T(p,q)\cup J_{0,a})$ is compressible or boundary parallel, by \refthm{LeeWithConverse}. We rule out each of these possibilities.

If $T$ is boundary parallel in $S^3-(T(p,q)\cup J_{0,a})$, then it bounds a solid torus $V$ in $S^3$ containing one of $T(p,q)$ or $J_{0,a}$, with that link component isotopic to a longitude on $\bdy V$. Because $T$ is essential in the complement of the full link $T(p,q)\cup J_{0,a}\cup J_{0,b}$, $V$ must also contain $J_{0,b}$.

Suppose first that $V$ contains $T(p,q)$ and $J_{0,b}$, and $T(p,q)$ is the core of $V$. That is, $T=\bdy V$ is boundary parallel to the boundary of a regular neighbourhood of $T(p,q)$. Project to $S^3-(T(p,q)\cup J_{0,b})$. The torus $T$ cannot be boundary parallel in that link complement because it bounds a solid torus containing both link components. Thus since it is inessential in $S^3-(T(p,q)\cup J_{0,b})$, it is compressible in $S^3-(T(p,q)\cup J_{0,b})$. Because $T(p,q)$ is the core of $V$, a compressing disk for $T$ lies on the side of $V$ that does not contain $T(p,q)$ and $J_{0,b}$. But then $V$ must be unknotted in $S^3$. This contradicts the fact that $T(p,q)$ is the core of $V$.

Suppose then that $V$ contains $J_{0,a}$ and $J_{0,b}$, and $J_{0,a}$ is the core of $V$. We know that $T$ is compressible or boundary parallel in $S^3-(T(p,q)\cup J_{0,b})$. If $T$ were compressible, then since $J_{0,b}$ lies inside $V$, $J_{0,b}$ would have linking number zero with $T(p,q)$, which is a contradiction. Thus it is boundary parallel in $S^3-(T(p,q)\cup J_{0,b})$. It is not boundary parallel to a regular neighbourhood of $T(p,q)$, since both $J_{0,a}$ and $J_{0,b}$ lie on the opposite side of $T$ in $V$, and so then $T$ would remain boundary parallel to a regular neighbourhood of $T(p,q)$ in the larger link $S^3-(T(p,q)\cup J_{0,a}\cup J_{0,b})$, contradicting the fact that it is essential there. So $T$ is boundary parallel to $J_{0,b}$ and to $J_{0,a}$, on the same side. The solid torus in $S^3$ containing $J_{0,b}$ and $J_{0,a}$ therefore has linking number $b$ with $T(p,q)$, and has linking number $a$ with $T(p,q)$. Since $b\neq a$, this is a contradiction.

So $T$ must be compressible in $S^3-(T(p,q)\cup J_{0,a})$. Similarly, 
$T$ cannot be boundary parallel in $S^3-(T(p,q)\cup J_{0,b})$. So $T$ is also compressible in $S^3-(T(p,q)\cup J_{0,b})$.

A compressing disc $D$ for $T$ in $S^3-(T(p,q)\cup J_{0,a})$ must intersect $J_{0,b}$. Surger $T$ along $D$ to obtain a sphere. By irreducibility of $S^3-(T(p,q)\cup J_{0,a})$, this sphere bounds a ball disjoint from $T(p,q)\cup J_{0,a}$. If the disc $D$ lies on the outside of the ball, then $T$ contains the ball, $D$, and the link component $J_{0,b}$ on one side (a solid torus), and contains $T(p,q)\cup J_{0,a}$ on the other side. If the disc $D$ lies on the inside of the ball, then after undoing the surgery along $D$, the ball becomes a knot complement in $S^3$ with boundary $T$, and $J_{0,b}$, $T(p,q)$, and $J_{0,a}$ all lie in the solid torus $V$ on the opposite side of $T$.

Apply the above argument to $T(p,q)\cup J_{0,b}$. Again $J_{0,a}$ must intersect a compressing disc for $T$ in $S^3-(T(p,q)\cup J_{0,b})$ and so again either $T$ separates $J_{0,a}$ from $T(p,q)\cup J_{0,b}$ or all link components lie in a solid torus. Because $T$ cannot separate $T(p,q)$ from both $J_{0,a}$ and $J_{0,b}$, we must be in the latter case. So $T(p,q)$, $J_{0,a}$, and $J_{0,b}$ lie in a knotted solid torus in $S^3$, with $T(p,q)\cup J_{0,a}$ lying in one ball inside that solid torus, and $T(p,q)\cup J_{0,b}$ lying in another ball inside that solid torus. This gives a contradiction to \reflem{ComponentsInBall}.
\end{proof}

\begin{proposition}\label{Prop:Atoroidal}
Let $p$ and $q$ be relatively prime integers such that $1<q<p$, and let $r_1, \dots, r_n$ be integers such that $1<r_1<\dots<r_n<p$, and suppose that no $r_i$ is a multiple of $q$. Then $T(p,q) \cup J_{0,r_1} \cup \dots \cup J_{0,r_n}$ is atoroidal. 
\end{proposition}

\begin{proof}
Suppose by way of contradiction that $T$ is an essential torus in the complement of $K=T(p,q)\cup J_{0,r_1}\cup \dots \cup J_{0,r_n}$. If $n=1$, then by \refthm{LeeWithConverse}, $r_1$ must be a multiple of $q$, giving an immediate condradiction. So assume $n>1$. By \refthm{LeeWithConverse}, $T$ cannot be essential when projected into $S^3-(T(p,q)\cup J_{0,r_k})$ for any $k$. Moreover, \reflem{TwoJ} implies $T$ cannot be essential when projected into $T(p,q)\cup J_{0,r_k}\cup J_{0,r_j}$ for any $r_j\neq r_k$. Thus if $n=2$, we have a contradiction.

If $n>2$, then $T$ is either compressible or boundary parallel in each of the link complements $S^3-(T(p,q)\cup J_{0,r_k}\cup J_{0,r_j})$ by \reflem{TwoJ}. We consider separately the cases that $T$ is compressible in every one of these link complements, and the case that $T$ is boundary parallel in one of them.

\smallskip

\underline{Case 1.}
Suppose first that $T$ is compressible in every one of these link complements. 
Surgering $T$ along a compressing disc $D$ for $T$ in $T(p,q)\cup J_{0,r_k}\cup J_{0,r_j}$ gives a sphere, which, by irreducibility, bounds a ball $B$ disjoint from the 3-component link. There are two cases: If $D$ lies on the outside of $B$, then $T$ bounds a solid torus disjoint from $T(p,q)\cup J_{0,r_k}\cup J_{0,r_j}$. If $D$ lies on the inside of $B$, then $T$ bounds a solid torus that contains $T(p,q)\cup J_{0,r_k}\cup J_{0,r_j}$.

Suppose the first case happens for a link $T(p,q)\cup J_{0,r_1}\cup J_{0,r_2}$, with compressing disc $D$ lying outside the ball disjoint from the 3-component link. Because $T$ is incompressible in the full link complement, some other component $J_{0,r_3}$ must intersect $D$. Thus it lies on the opposite side of $T$ from $T(p,q)$. But then consider $T(p,q)\cup J_{0,r_1}\cup J_{0,r_3}$. Surgering along a compression disc for $T$ in this link complement gives a sphere with $T(p,q)\cup J_{0,r_1}$ lying on one side and $J_{0,r_3}$ on the other, contradicting irreducibility of the 3-component link complement. 

Thus the second case happens for all choices of links $T(p,q)\cup J_{0,r_k}\cup J_{0,r_j}$, i.e.\ the 3-component link lies inside a ball within a solid torus $V$ bounded by $T$. But this is a contradiction of \reflem{ComponentsInBall}.

\smallskip

\underline{Case 2.}
It follows that $T$ is boundary parallel in some $S^3-(T(p,q)\cup J_{0,r_i}\cup J_{0,r_j})$. Thus $T$ bounds a solid torus $V \subset S^3$ containing either $T(p,q)$ or one of the $J_{0,r_k}$, for $k=i,j$, on one side, with $\bdy V$ parallel to the link component it bounds. Suppose $T$ is boundary parallel to $T(p,q)$. Because $T$ is essential in the complement of $K$, the solid torus $V$ must contain another link component $J_{0,r_\ell}$. By \reflem{TwoJ}, $T$ is inessential in the complement of $T(p,q)\cup J_{0,r_k}\cup J_{0,r_\ell}$, and $T(p,q)$ lies on one side of $T$ and $J_{0,r_k}$ on the other. If $T$ is compressible, then by surgering $T$ along a compressing disc, we obtain a sphere containing $T(p,q)$ on one side and $J_{0,r_k}$ on the other, implying that the linking number between $T(p,q)$ and $J_{0,r_k}$ is zero, a contradiction. Hence $T$ cannot be compressible in this link complement.

Hence it must be boundary parallel here.
Because $T(p,q)$ and $J_{0,r_\ell}$ lie in $V$, it cannot be boundary parallel on that side. Hence it is boundary parallel to $J_{0,r_k}$ on its opposite side. This implies that $T$ is the boundary of a neighbourhood of the unknot and therefore $T(p,q)$ is unknotted, which is a contradiction.

So $T$ is boundary parallel to $J_{0,r_k}$ for $k=i$ or $j$. Therefore, $T$ is trivial. Again because $T$ is essential in the complement of $K$, the solid torus $V$ must
contain another $J_{0,r_\ell}$. Then $T$ is inessential in $S^3-(T(p,q)\cup J_{0,r_k}\cup J_{0,r_\ell})$ by \reflem{TwoJ}, but contains $J_{0,r_\ell}$ and $J_{0,r_k}$ in 
a solid torus on one side, and $T(p,q)$ on the other. The wrapping numbers of $J_{0,r_\ell}$ and $J_{0,r_k}$ in $V$ are non-zero, and the wrapping number
of $T(p,q)$ in $S^3-V$ is also non-zero, otherwise the linking number between $T(p,q)$ and $J_{0,r_\ell}$ or $J_{0,r_k}$ would be zero, a contradiction. 
Thus, $T$ is incompressible in $S^3-(T(p,q)\cup J_{0,r_k}\cup J_{0,r_\ell}).$ As $T$ can't be boundary parallel to $T(p,q)$, $T$ would be essential 
in $S^3-(T(p,q)\cup J_{0,r_k}\cup J_{0,r_\ell})$, a contradiction.
\end{proof}

\begin{proposition}\label{Prop:Annular}
Let $p,q$ be relatively prime integers with $1<q<p$, and let $r_1, \dots, r_n$ be integers such that $1<r_1<\dots<r_n<p$ for $j=1,\dots,n$. Suppose no $r_j$ is a multiple of $q$. Then 
$K=T(p,q)\cup J_{0,r_1}\cup \dots \cup J_{0,r_n}$ is anannular. 
\end{proposition}

\begin{proof}
By work of Lee~\cite[Lemma~5.2]{Lee:Unknotted}, there is no annulus embedded in $S^3-N(T(p,q)\cup J_{0,a})$ with one boundary component on $\bdy N(T(p,q))$ and one on $\bdy N(J_{0,a})$, for any $a\neq p,q$. 

We now show that if $J_{0,a} \neq J_{0,b}$, then there can be no annulus embedded in $S^3-N(K)$ with one boundary component on $\bdy N(J_{0,a})$ and one on $\bdy N(J_{0,b})$. Because $J_{0,a}$ and $J_{0,b}$ are unknots, any such annulus $A$ would have one boundary component $\bdy_1 A$ a torus knot on $\bdy N(J_{0,a})$ and the other boundary component $\bdy_2 A$ a torus knot on $\bdy N(J_{0,b})$, and these torus knots would be ambient isotopic. Because $N(J_{0,a})$ and $N(J_{0,b})$ have linking number zero, $\bdy_1 A\subset \bdy N(J_{0,a})$ has linking number zero with $N(J_{0,b})$. It follows that $\bdy_2 A$, since it is ambient isotopic to $\bdy_1 A$, is a longitude of $\bdy N(J_{0,b})$. Similarly, $\bdy_1 A$ is a longitude of $\bdy N(J_{0,a})$. Thus each boundary component of the annulus $A$ is a longitude, and $J_{0,a}$ and $J_{0,b}$ are isotopic through this annulus in $S^3-K$. This is a contradiction: $J_{0,a}$ has linking number $a$ with $T(p,q)$, and $J_{0,b}$ has linking number $b$ with $T(p,q)$, and $a\neq b$. 

Thus any essential embedded annulus $A$ in the exterior of $K$ has both boundary components on the same link component.

Consider first that $\bdy A \subset \bdy N(T(p, q))$. The annulus $A$ is not essential in $S^3-N(T(p,q))$, since the exterior of a torus knot has just one essential annulus by work of Tsau~\cite{Tsau}, and by \cite[Lemma~5.1]{Lee:Unknotted}, that essential annulus would be punctured by $J_{0,r_j}$. Thus, $A$ is compressible, boundary compressible, or boundary parallel in $S^3-N(T(p,q))$. Observe that a boundary compressible annulus is in fact boundary parallel, using the fact that $S^3-N(T(p,q))$ is irreducible. 

First suppose $A$ is boundary parallel to an annulus $B$ in $S^3-N(T(p,q))$.  Then $A\cup B$ bounds a solid torus $V$ in $S^3-N(T(p,q))$.
Since  $A$ is not boundary parallel in $S^3-K$, at least one $J_{0,a}$ must be inside $V$. In addiction, $J_{0,a}$ has wrapping number greater than zero in $V$ otherwise $T(p, q)$  and $J_{0,a}$ would have linking number equal to zero, a contradiction. Suppose $\bdy V$ is boundary parallel to $J_{0,a}$. If $\bdy V \cap \bdy N(T(p,q))$ is a meridional annulus on $N(T(p,q))$, then $J_{0,a}$ would have linking number one with $T(p,q)$, which is impossible by choice of $a$. If $\bdy V \cap \bdy N(T(p,q))$ is not meridional, then $V$ forms a nontrivial knot in $S^3$, and $J_{0,a}$ lies at its core, contradicting the fact that $J_{0,a}$ is trivial. So $\bdy V$ is not boundary parallel to $J_{0,a}$. It follows that the wrapping number of $J_{0,a}$ is strictly greater than one. But then $\bdy V$ is an essential torus in $S^3-(T(p,q)\cup J_{0,a})$, contradicting \refthm{LeeWithConverse}. 

Assume now that $A$ is compressible in $S^3-N(T(p,q))$. Then, there is a compression disk $D$ for $A$ in $S^3-N(T(p,q))$. Surgering $A$ along $D$ yields two discs, $D_1$ and $D_2$, such that $\partial A = \partial D_1 \cup \partial D_2$.
Since $S^3-N(T(p,q))$ is boundary irreducible, $\partial D_i$ bounds a disk $E$ on $\partial N(T(p,q))$. Thus, by pushing $E$ slightly off of $\bdy N(L)$ in $S^3-N(K)$, we obtain a compressing disc for $A$ in $S^3-N(K)$, which contradicts our assumption that $A$ is essential. Therefore $A$ is not boundary parallel. It follows that $A$ cannot have both boundary components on $\bdy N(T(p,q))$.
  
Consider now that $A$ has both boundary components on $\bdy N(J_{0,a})$ for some $a=r_i$. Since the solid torus has no essential annuli, $A$ is not essential in $S^3-N(J_{0,a})$. Again observe that $A$ is either compressible or boundary parallel in $S^3-N(J_{0,a})$. 

Suppose $A$ is boundary parallel to an annulus $B$ on $\bdy N(J_{0,a})$ in $S^3-J_{0,a}$. Then $A\cup B$ bounds a solid torus $V$ in $S^3-J_{0,a}$. Since $A$ is not boundary parallel in $S^3-N(K)$, at least one component of $K$ must be inside $V$. Consider first that $J_{0,b} \subset V$ for some $b$. If $T(p,q)$ is not contained in $V$, then expanding a neighbourhood of $J_{0,a}$ to contain $V$ gives a trivial solid torus containing $J_{0,a}$ and $J_{0,b}$. By consideration of linking number and wrapping numbers, its boundary $\bdy V$ must be essential in $S^3-(T(p,q)\cup J_{0,a}\cup J_{0,b})$, contradicting \refprop{Atoroidal}. Assume then that $T(p, q)$ is contained in $V$. The core of $V$ can be isotoped to lie on the boundary of a neighbourhood of the unknot $J_{0,a}$, hence it forms a torus knot. Observe that the linking number of any curve $C$ outside of $V$ with $T(p,q)$ is then obtained by multiplying the linking number of $C$ with $J_{0,a}$, the linking number of the core of $V$ with $J_{0,a}$, and the winding number of $T(p,q)$ within $V$. 
But any curve $J_{0,c}$ has zero linking number with $J_{0,a}$, but nonzero linking number with $T(p,q)$. Thus all other components of $K$ must lie within $V$. 
If $T(p,q)$ has wrapping number zero in $V$, then the linking number of $T(p,q)$ and $J_{0,a}$ would be zero, which is not possible. Furthermore, the core of $V$ cannot be isotopic to a simple longitude of $N(J_{0,a})$, for the same reason. If the core of $V$ is isotopic to a meridian of $N(J_{0,a})$ then because all link components lie inside $V$, a compressing disc for the solid torus $S^3-(N(J_{0,a})\cup V)$ is a boundary compression disc for $A$ in $S^3-N(K)$, contradicting our assumption that $A$ is essnential in $S^3-N(K)$. 
But then the core of $V$ is a nontrivial torus knot on $\bdy N(J_{0,a})$, hence $\bdy V$ is an essential torus in $S^3-K$, contradicting \refprop{Atoroidal}.

The remaining case is that $A$ has both boundary components on $\bdy N(J_{0,a})$ and is compressible in $S^3-N(J_{0,a})$. Then there is a compression disk $D$ for $A$ in $S^3-N(J_{0,a})$. Surgering $A$ along $D$ yields two discs, $D_1$ and $D_2$, such that $\partial A = \partial D_1 \cup \partial D_2$.
If one $\partial D_i$ bounds a disk $E$ on $\bdy N(J_{0,a})$, then by considering a disc with boundary in $A$ close to $E$, we see that $A$ is also compressible in $S^3-N(K)$, a contradiction. So suppose that $\partial D_i$ does not bound a disk on $\bdy N(J_{0,a})$. Then both $\partial D_1$ and $\partial D_2$ are either isotopic to the longitude or meridian of $N(J_{0,a})$ implying that $A$ is also boundary parallel in $S^3-N(J_{0,a})$. Thus, we have a contradiction to the previous paragraph. 
\end{proof}

\begin{theorem}\label{Thm:AugHyperbolic}
Let $p,q$ be relatively prime integers with $1<q<p$. Let $a_1, \dots, a_n$ be integers such that $1<a_1<\dots< a_n<p$. Consider the link
\[ K=T(p,q)\cup J_{0,a_1}\cup \dots J_{0,a_n}. \]
\begin{enumerate}
\item If no $a_i$ is a multiple of $q$, then $K$ is hyperbolic.
\item If all $a_i>q$ are multiples of $q$, it is satellite, and remains satellite under $1/b_i$ Dehn filling on $J_{0,a_i}$, for any integers $b_i>0$, $i=1, \dots, n$. After Dehn filling, the companion and pattern are T-links.
\end{enumerate}
\end{theorem}

\begin{proof}
By Thurston's geometrisation theorem for Haken manifolds~\cite{Thurston:Bulletin}, $S^3-K$ is hyperbolic if and only if it is irreducible, boundary irreducible, atoroidal, and anannular. By \reflem{Irreducible}, it is irreducible and boundary irreducible. By \refprop{Atoroidal}, it is atoroidal if no $r_j$ is a multiple of $q$, and by \refprop{Annular} it is also anannular in this case. This proves statement~(1) on hyperbolicity.

The statement~(2) on satellite links follows from \refthm{SatelliteIntro}, since $1/b_i$ Dehn filling produces a T-link with all full twists as in the statement of that theorem. 
\end{proof}

\begin{corollary}\label{Cor:FullTwistConj}
Let $p,q$ be relatively prime integers with $1<q<p$, and let $a_1, \dots, a_n$ and $b_1, \dots, b_n$ be integers such that $1<a_1<\dots<a_n<p$ and $b_i>0$, and no $a_i$ is a multiple of $q$. Then there exists $B\gg 0$ such that if each $b_i>B$, then
$T((a_1, a_1b_1), \dots, (a_n,a_nb_n), (p,q))$ is hyperbolic.
\end{corollary}

\begin{proof}
By \refthm{AugHyperbolic}, a parent link before filling is hyperbolic if the $a_i$ satisfy the hypotheses of the corollary. Under high Dehn filling, if the parent is hyperbolic, the Dehn filling remains hyperbolic by Thurston's hyperbolic Dehn filling theorem~\cite{thurston:notes}.
\end{proof}

\section{Toroidal T-links without twists}\label{Sec:NoTwistToroidal}

We have found infinitely many toroidal T-links obtained by adding full twists on strands of a torus link, and infintely many more T-links obtained by full twisting that cannot be toroidal. Our results so far give evidence for the Lorenz satellite conjecture, \refconj{LorenzSatellite}, but they only apply to T-links obtained by full twisting.

In this section we give a family of T-links that are not obtained by adding full or even half-twists to torus links in any obvious manner, but are still toroidal, and still satisfy \refconj{LorenzSatellite}.

\begin{theorem}\label{Thm:NoTwistToroidal}
  Choose positive integers as follows. Let $c\geq 3$, $r\geq 2$, $1\leq k\leq r-1$. Let $2\leq a_1 < \dots < a_n \leq r-k$, and let $b_1, \dots, b_n>0$. Then the T-link
  \[ T((a_1,b_1),\dots,(a_n,b_n),(rc-k, r-k), (rc, r(c-2)+k)) \]
  is a satellite link with companion the torus knot $T(c,c-1)$ and pattern the T-link
  \[ T((a_1,b_1), \dots, (a_n,b_n),(r-k, r-k),(r,r(c-1)^2 + r(c-2)+k)). \]
\end{theorem}

\begin{proof}
Let $\tau$ be the braid $(a_1,b_1)*\dots*(a_n,b_n)$. The T-link under consideration is the closure of the braid
\[  \tau*(rc-k,r-k)*(rc,r(c-2)+k), \]
which is equivalent to the closure of the braid
\begin{equation}\label{Eqn:BraidNoTwist}
 (rc,r(c-2)+k)*\tau*(rc-k,r-k),
\end{equation}
after conjugating by $\tau*(rc-k,r-k)$, or alternatively, moving the crossings corresponding to the braid $\tau*(rc-k,r-k)$ around the braid closure.

The theorem is proved by drawing the braid corresponding to the torus knot $T(c,c-1)$ as a thick tube, and then carefully placing the braid of equation~\eqref{Eqn:BraidNoTwist} inside the tube, as shown in \reffig{SatelliteNoTwists}.

\begin{figure}
  \centering
  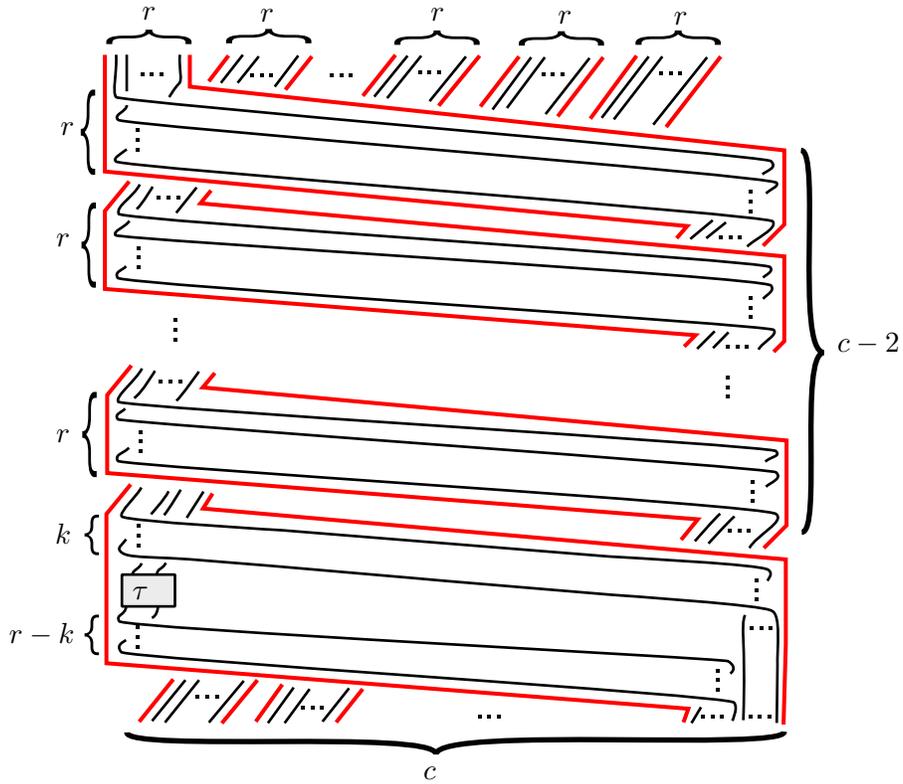
  \caption{The closed braid of equation~\eqref{Eqn:BraidNoTwist} fits inside a tube shown in red.}
  \label{Fig:SatelliteNoTwists}
\end{figure}

Precisely, there are $c$ tubes across the top of the braided tube corresponding to the braid $(c,c-1)$. In each tube, place $r$ strands. For the first $c-2$ overstrands running left to right in the braid $(c,c-1)$, place the first $r(c-2)$ overstrands running left to right of the braid $(rc,r(c-2)+k)$.

The final tube running left to right is different. There are $r$ strands in the tube. The first $k$ of these run all the way left to all the way right, completing the braid $(rc,r(c-2)+k)$. Just below these overstrands are $r-k$ strands on the left of the braid. These run into the braid $\tau$, as shown in \reffig{SatelliteNoTwists}. From there, they form the $r-k$ strands running all the way across $c-1$ of the tubes, forming the braid $(r(c-1)+r-k, r-k)=(rc-k,r-k)$. This completes the braid of equation~\eqref{Eqn:BraidNoTwist}.

Take the closure of the braid, and the closure of the tubes, each containing $r$ strands of the braid. This gives the T-link of the theorem, contained in a closed solid torus. Hence the T-link is satellite. The companion is the torus knot $T(c,c-1)$, which is the torus knot obtained by taking the closure of the tubes.

The pattern is a closed braid determined as follows. Begin in the top left tube. There are $r$ strands running into the braid. As the tube runs left to right across the braid, those $r$ strands form a full twist, each crossing once over all the other strands. This happens $c-2$ times, picking up $(c-2)r$ strands crossing fully over $r$ strands of the braid within the tube. In the $(c-1)$-th tube, $k$ strands run over all the $r$ strands, completing the braid $(r,r(c-2)+k)$. The $r-k$ strands on the left then run into the braid $\tau$. Upon exiting, $r-k$ strands form a full twist, which is the braid $(r-k,r-k)$. Putting it together, this is the closure of the braid
\[ (r,r(c-2)+k)*\tau*(r-k,r-k).\]
However, this braid is lying inside a tube about the torus knot $T(c,c-1)$. To obtain the pattern, we apply a homeomorphism taking the tube to the unknotted solid torus, taking a standard longitude of $T(c,c-1)$ to a standard longitude of the unknot. By linking number considerations, the effect is to add an additional $(c-1)^2$ full twists to the braid that makes up the pattern, giving the pattern as the closure of the braid
\[ (r,r(c-2)+k)*\tau*(r-k,r-k)*(r,r(c-1)^2), \]
now within a tube about the unknot rather than the $T(c,c-1)$ torus knot.
After braid conjugation, this is equivalent to the T-link pattern claimed in the statement of the theorem.
\end{proof}

\begin{theorem}\label{Thm:BirmanWilliamsExtend}
Let $1<a_1<\dots<a_n$ and $b_1, \dots, b_n$ be positive integers such that $K_1=T((a_1,b_1),\dots,(a_n,b_n))$ is a knot. Let $1<c_1<\dots<c_m$ be integers, and $d_1, \dots, d_m$ be positive integers, and set $B=b_1+ \dots +b_n$, and set $D$ to be the number of crossings of $K_1$, or $D=\sum_{i=1}^n (a_i-1)b_i$. Then the satellite link with companion the T-link $K_1$ 
and pattern the T-link $K_2 = T((c_1, d_1),\dots,(c_m,d_m+c_mB+c_mD))$
is the T-link
\[ T= T((c_1,d_1), \dots, (c_{m-1},d_{m-1}), (c_m,d_m), (c_ma_1,c_mb_1), \dots, (c_ma_n,c_mb_n)).\]
\end{theorem}

\begin{proof}
Consider the braid of $K_1= T((a_1,b_1),\dots,(a_n,b_n))$. Take a tube about this braid, similar to the red tube in \reffig{SatelliteNoTwists}. In the top left corner within the tube, add the braid corresponding to the T-link $T((c_1,d_1), \dots, (c_m,d_m))$. This braid has $c_m$ strands running into the top, and $c_m$ strands running out the bottom. Let the strands running out the bottom follow the tube about the knot $K_1$. To ensure this has the presentation of a T-link, for each overstrand of $K_1$ we insert one full twist, with half of the crossings on the far left, and half on the far right, similar to what is shown for each of the first $c-2$ overstrands \reffig{SatelliteNoTwists}. Elsewhere, let the $c_m$ strands run through the tube without crossing. This gives the T-link of the theorem.

The companion is $K_1$ by construction. The pattern is obtained by taking a homeomorphism of the tube about $K_1$ to the tube about the unknot, sending a standard longitude to a standard longitude. As in the previous theorem, this adds $D$ additional full twists, giving the pattern as a closed braid, with the braid equal to
\[ (c_1,d_1)* \cdots * (c_m,d_m)* (c_m,c_mB)* (c_m,c_mD). \]
This is the claimed pattern.
\end{proof}

\begin{corollary}\label{Cor:BWIntro}
  For any two one-component T-links $K_1$ and $K_2$, there exists a satellite T-link $K$ such that after cutting $S^3-K$ along an essential torus, the components consist of the complement of $K_1$ in $S^3$, and the complement of $K_2$ in a solid torus.
\end{corollary}

\begin{proof}
For given T-links $K_1$ and $K_2$, \refthm{BirmanWilliamsExtend} constructs a T-link with companion $K_1$ and pattern obtained from $K_2$ by adding $B+D$ additional full twists. Because the homeomorphism type of a link in a solid torus is unaffected by the additional twisting, we obtain the result. 
\end{proof}

\subsection*{Acknowledgements}
We thank Norman Do for his questions leading to Question~\ref{Ques:NormQuestion}. We thank Ilya Kofman and Sangyop Lee for helpful discussions. 
The authors were supported in part by grants from the Australian Research Council.

\bibliographystyle{amsplain}  

\bibliography{references}

\end{document}